\documentclass{amsart}

\usepackage{amssymb,amsmath,amsthm,latexsym,amscd}

\theoremstyle{definition}
\newtheorem{definition}{Definition}[section]
\newtheorem{example}[definition]{Example}
\newtheorem{remark}[definition]{Remark}

\theoremstyle{plain}
\newtheorem{lemma}[definition]{Lemma}
\newtheorem{proposition}[definition]{Proposition}
\newtheorem{theorem}[definition]{Theorem}

\newtheorem{conjecture}[definition]{Conjecture}

\newcommand{\nn}{\!\!}

\begin{document}


\title{Jordan Triple Disystems}

\author{Murray R. Bremner}

\address{Department of Mathematics and Statistics, University of Saskatchewan, Canada}

\email{bremner@math.usask.ca}

\author{Ra\'ul Felipe}

\address{CIMAT, Centro de Investigaci\'on en Matem\'aticas, Guanajuato, M\'exico}

\email{raulf@cimat.mx}

\author{Juana S\'anchez-Ortega}

\address{Departmento de \'Algebra, Geometr\'ia y Topolog\'ia, Universidad de M\'alaga, Espa\~na}

\email{jsanchez@agt.cie.uma.es}

\begin{abstract}
We take an algorithmic and computational approach to a basic problem in abstract algebra:
determining the correct generalization to dialgebras of a given variety of nonassociative algebras.
We give a simplified statement of the KP algorithm introduced by Kolesnikov and Pozhidaev 
for extending polynomial identities for algebras to corresponding identities for dialgebras. 
We apply the KP algorithm to the defining identities for Jordan triple systems to obtain a new 
variety of nonassociative triple systems, called Jordan triple disystems.
We give a generalized statement of the BSO algorithm introduced by Bremner and S\'anchez-Ortega for extending 
multilinear operations in an associative algebra to corresponding operations in an associative dialgebra.
We apply the BSO algorithm to the Jordan triple product and use computer algebra to verify that the 
polynomial identities satisfied by the resulting operations coincide with the results of the KP algorithm;
this provides a large class of examples of Jordan triple disystems.
We formulate a general conjecture expressed by a commutative diagram relating the output of the KP and 
BSO algorithms.
We conclude by generalizing the Jordan triple product in a Jordan algebra to operations in a Jordan dialgebra; 
we use computer algebra to verify that resulting structures provide further examples of Jordan triple disystems.  
For this last result, we also provide an independent theoretical proof using Jordan structure theory.
\end{abstract}

\maketitle


\section{Introduction}

The theory of Jordan algebras, originally motivated by potential applications to quantum physics,
was initiated by Jordan, von Neumann and Wigner \cite{Jordan},
and the theory of Jordan triple systems was initiated by Jacobson \cite{JacobsonJTS}.
Standard references on these topics are
Braun and Koecher \cite{Braun},
Jacobson \cite{Jacobson},
McCrimmon \cite{McCrimmon},
Neher \cite{Neher} and Loos \cite{Loos};
for applications to geometry and analysis see Faraut et al. \cite{Farautetal}.
The concept of an associative dialgebra was introduced by Loday 
\cite{LodayDialgebras};
the generalization of the Lie bracket produces Lie dialgebras (also called Leibniz algebras)
which were first introduced by Cuvier \cite{Cuvier} and Loday \cite{LodayLeibniz}.
Numerous authors have considered other varieties of nonassociative dialgebras;
in particular, the generalization of the Jordan product produces 
Jordan dialgebras (also called quasi-Jordan algebras),
which have been studied by
Vel\'asquez and Felipe \cite{VelasquezFelipe1,VelasquezFelipe2},
Kolesnikov \cite{Kolesnikov}, 
Pozhidaev \cite{Pozhidaev1,Pozhidaev2},
Bremner \cite{Bremner}, 
Bremner and Peresi \cite{BremnerPeresi2},
and Voronin \cite{Voronin}.

The purpose of the present paper is to introduce a new variety of triple systems of Jordan type,
which we call Jordan triple disystems.
The relation between Jordan triple disystems and associative dialgebras is analogous to the relation between 
Jordan triple systems and associative algebras. 

Section \ref{preliminaries} recalls basic definitions for associative dialgebras.
Section \ref{KP algorithm} presents a simplified statement of the general Kolesnikov-Pozhidaev (KP) algorithm
for converting an arbitrary variety of multioperator algebras into a variety of dialgebras.
We recall how this algorithm can be applied to the defining identities for Jordan algebras to obtain the
variety of Jordan dialgebras.
In Section \ref{jordantripledisystems} we apply the KP algorithm to the defining identities for 
Jordan triple systems; 
we obtain a system of polynomial identities which define our new variety of Jordan triple disystems.
Section \ref{associativediproducts} presents a generalized statement of the Bremner-S\'anchez-Ortega (BSO)
algorithm for extending multilinear operations in an associative algebra to an associative dialgebra.
We apply the BSO algorithm to the Jordan triple product and use computer algebra to verify that the 
resulting polynomial identities coincide with the results of the KP algorithm; we therefore obtain a large 
class of examples of Jordan triple disystems.
In Section \ref{conjecturesection} we formulate a general conjecture expressed by a commutative diagram 
relating the output of the KP and BSO algorithms.
Section \ref{Jordandialgebrasection} generalizes the Jordan triple product in a Jordan algebra to operations 
in Jordan dialgebras; 
we use computer algebra to verify that resulting structures provide another large class of examples of 
Jordan triple disystems.  
We also provide an independent theoretical proof of this last result using Jordan structure theory.


\section{Dialgebras} \label{preliminaries}

\subsection{Dialgebras and Leibniz algebras}

Dialgebras were introduced by Loday \cite{LodayLeibniz, LodayDialgebras, LodaySurvey} 
to provide a natural setting for Leibniz algebras, a ``noncommutative'' generalization 
of Lie algebras.

\begin{definition} (Cuvier \cite{Cuvier}, Loday \cite{LodayLeibniz})
A {\bf Leibniz algebra} is a vector space $L$ together with a bilinear map $L
\times L \to L$, denoted $(a,b) \mapsto [a,b]$ and called the {\bf Leibniz
bracket}, satisfying the {\bf Leibniz identity}: 
  \[
  [ [ a, b ], c] \equiv [ [ a, c ], b ] + [ a, [ b, c ] ].
  \]
If $[ a, a ] \equiv 0$ then the Leibniz identity is the Jacobi identity and $L$
is a Lie algebra.
\end{definition}

Every associative algebra becomes a Lie algebra if the associative product is
replaced by the Lie bracket. Loday introduced the notion of dialgebra which
gives, by a similar procedure, a Leibniz algebra: one replaces $ab$ and $ba$ by
two distinct operations, so that the resulting bracket is not necessarily
skew-symmetric.

\begin{definition}
An \textbf{associative dialgebra} is a vector space $D$ with two bilinear operations
$\dashv\colon D \times D \to D$ and $\vdash\colon D \times D \to D$,
the \textbf{left} and \textbf{right} products, satisfying the left and right bar identities,
and left, right and inner associativity:
\begin{alignat*}{3}
&
( a \dashv b ) \vdash c
\equiv
( a \vdash b ) \vdash c,
&\quad
&
a \dashv ( b \dashv c )
\equiv
a \dashv ( b \vdash c ),
\\
&
( a \dashv b ) \dashv c
\equiv
a \dashv ( b \dashv c ),
&\quad
&
( a \vdash b ) \vdash c
\equiv
a \vdash ( b \vdash c ),
&\quad
&
( a \vdash b ) \dashv c
\equiv
a \vdash ( b \dashv c ).
\end{alignat*}
From a dialgebra we construct a Leibniz bracket by $a \dashv b - b \vdash a$.
\end{definition}

\subsection{Free dialgebras}

Loday has determined a basis for the free dialgebra.

\begin{definition} 
A \textbf{dialgebra monomial} on  a set $X$ is a product $x = a_1 a_2 \cdots
a_n$ where $a_1, \ldots, a_n \in X$ with some placement of parentheses and some
choice of operations. The {\bf center} of $x$ is defined inductively: if $n =
1$ then $c(x) = x$; if $n \ge 2$ then $x = y \dashv z$ or $x = y \vdash z$ and
we set $c( y \dashv z ) = c(y)$ or $c( y \vdash z ) = c(z)$.
\end{definition}

\begin{lemma} \emph{(Loday \cite{LodaySurvey})}
If $x = a_1 a_2 \cdots a_n$ is a monomial with $c(x) = a_i$ then
$x$ is determined by the order of its factors and the position of its center:
  \[
  x =
  ( a_1 \vdash \cdots \vdash a_{i-1} )
  \vdash a_i \dashv
  ( a_{i+1} \dashv \cdots \dashv a_n ).
  \]
\end{lemma}

\begin{definition}
The right side of the last equation is the {\bf normal form} of $x$ and is
abbreviated by the \textbf{hat notation} $a_1 \cdots a_{i-1} \widehat{a}_i
a_{i+1} \cdots a_n$.
\end{definition}

\begin{lemma} \emph{(Loday \cite{LodaySurvey})}
The set of monomials $a_1 \cdots a_{i-1} \widehat{a}_i a_{i+1} \cdots a_n$ in
normal form with $1 \le i \le n$ and $a_1, \dots, a_n \in X$ forms a basis of
the free dialgebra on $X$.
\end{lemma}


\section{The Kolesnikov-Pozhidaev algorithm} \label{KP algorithm}

This algorithm, introduced by Kolesnikov \cite{Kolesnikov} and Pozhidaev \cite{Pozhidaev2},
converts a multilinear polynomial identity of degree $d$ for an $n$-ary operation 
into $d$ multilinear identities of degree $d$ for $n$ new $n$-ary operations. 

\begin{definition}

\textbf{KP Algorithm.}

Part 1:
We consider a multilinear $n$-ary operation, denoted by the symbol
  \begin{equation} \label{operation}
  \{-,-,\dots,-\} 
  \qquad
  \text{($n$ arguments)}.
  \end{equation}
Given a multilinear polynomial identity of degree $d$ in this operation, 
we describe the application of the algorithm to one monomial in the identity, 
and from this the application to the complete identity follows by linearity. 
Let $\overline{a_1 a_2 \dots a_d}$
be a multilinear monomial of degree $d$, where the bar denotes some placement of $n$-ary operation symbols. 
We introduce $n$ new $n$-ary operations, denoted by the same symbol but distinguished by subscripts:
  \begin{equation} \label{noperations}
  \{-,-,\dots,-\}_1,
  \quad
  \{-,-,\dots,-\}_2,
  \quad
  \dots,
  \quad
  \{-,-,\dots,-\}_n.
  \end{equation}
For each $i \in \{1, 2, \dots, d\}$ we convert the monomial $\overline{a_1 a_2 \dots a_d}$
in the original $n$-ary operation \eqref{operation} 
into a new monomial of the same degree $d$ in the $n$ new $n$-ary operations \eqref{noperations},
according to the following rule which is based on the position of $a_i$. 
For each occurrence of the original operation symbol in the monomial, 
either $a_i$
occurs within one of the $n$ arguments or not, and we have the following 
cases:
  \begin{itemize}
  \item
  If $a_i$ occurs within the $j$-th argument then we convert the original 
  operation symbol
  $\{\dots\}$ to the $j$-th new operation symbol $\{\dots\}_j$.
  \item
  If $a_i$ does not occur within any of the $n$ arguments, then either
    \begin{itemize}
    \item
    $a_i$ occurs to the left of the original operation symbol, in which case 
    we convert $\{\dots\}$ to the first new operation symbol $\{\dots\}_1$, or
    \item
    $a_i$ occurs to the right of the original operation symbol,
    in which case we convert $\{\dots\}$ to the last new operation 
    symbol $\{\dots\}_n$.
    \end{itemize}
  \end{itemize}
In this process, we call $a_i$ the central argument of the monomial.

Part 2:
In addition to the identities constructed in Part 1, we also include the following identities 
for all $i, j \in \{ 1, 2, \dots, n\}$ with $i \ne j$ 
and all $k, \ell \in \{ 1, 2, \dots, n\}$: 
  \begin{align*}  
  &
  \{ a_1, \dots, a_{i-1}, \{ b_1, \cdots, b_n \}_k, a_{i+1}, \dots, a_n \}_j
  \equiv
  \\ 
  &
  \{ a_1, \dots, a_{i-1}, \{ b_1, \cdots, b_n \}_\ell, a_{i+1}, \dots, a_n \}_j.
  \end{align*} 
This identity says that the $n$ new operations are interchangeable in 
the $i$-th argument of the $j$-th new operation when $i \ne j$.
\end{definition}

\begin{example}
The defining identities for associative dialgebras can be obtained by applying
the KP algorithm to the associativity identity, which we write in the form
$\{ \{ a, b \}, c \} \equiv \{ a, \{ b, c \} \}$.
The original operation produces two new operations $\{-,-\}_1$ and $\{-,-\}_2$.
Since associativity has degree 3, Part 1 produces three new identities of degree 3
by making $a$, $b$, $c$ in turn the central argument:
  \[
  \{ \{ a, b \}_1, c \}_1 {\equiv} \{ a, \{ b, c \}_1 \}_1, \,
  \{ \{ a, b \}_2, c \}_1 {\equiv} \{ a, \{ b, c \}_1 \}_2, \,
  \{ \{ a, b \}_2, c \}_2 {\equiv} \{ a, \{ b, c \}_2 \}_2,
  \] 
and Part 2 produces these two identities:
  \[
  \{ a, \{ b, c \}_1 \}_1 \equiv \{ a, \{ b, c \}_2 \}_1, 
  \quad
  \{ \{ a, b \}_1, c \}_2 \equiv \{ \{ a, b \}_2, c \}_2.
  \] 
If we revert to the standard notation by writing 
$a \dashv b = \{ a, b \}_1$ and $a \vdash b = \{ a, b \}_2$,
then these five identities are the defining identities for associative 
dialgebras.
\end{example}

\begin{definition}
A (linear) \textbf{Jordan algebra} is a vector space $J$ over a field of
characteristic $\ne 2$ with a bilinear operation $J \times J \to J$, denoted $a \circ b$,
satisfying commutativity and the {\bf Jordan identity} for all $a, b \in J$:
  \[
  a \circ b \equiv b \circ a,
  \qquad
  ( ( a \circ a ) \circ b ) \circ a \equiv ( a \circ a ) \circ ( b \circ a ).
  \]
\end{definition}

\begin{example} \label{jordanalgebraexample}
To apply the KP algorithm to Jordan algebras, we must start with multilinear identities,
so we linearize the Jordan identity to obtain
\begin{align*}
&
( ( a \circ c ) \circ b ) \circ d +
( ( a \circ d ) \circ b ) \circ c +
( ( c \circ d ) \circ b ) \circ a
\equiv
\\
&
( a \circ c ) \circ ( b \circ d ) +
( a \circ d ) \circ ( b \circ c ) +
( c \circ d ) \circ ( b \circ a ).
\end{align*}
(For a general discussion of linearization, see Zhevlakov et 
al.~\cite{Zhevlakov}, Chapter 1.)
We rewrite commutativity and the linearized Jordan identity using the symbol $\{-,-\}$:
\begin{align*}
&
\{ a, b \} - \{ b, a \} \equiv 0,
\\
&
\{ \{ \{ a, c \}, b \}, d \}
+
\{ \{ \{ a, d \}, b \}, c \}
+
\{ \{ \{ c, d \}, b \}, a \}
\\
&  \quad
-
\{ \{ a, c \}, \{ b, d \} \}
-
\{ \{ a, d \}, \{ b, c \} \}
-
\{ \{ c, d \}, \{ b, a \} \}
\equiv 0.
\end{align*}
The KP algorithm tells us to introduce two new operations $\{-,-\}_1$ 
and $\{-,-\}_2$.

Part 1:
Since commutativity has degree 2, we obtain two identities of degree 2
relating the two new operations:
\[
\{ a, b \}_1 - \{ b, a \}_2 \equiv 0,
\qquad
\{ a, b \}_2 - \{ b, a \}_1 \equiv 0,
\]
These identities are both equivalent to $\{ a, b \}_2 \equiv \{ b, a \}_1$:
the second operation is the opposite of the first.
Hence we can replace every occurrence of $\{-,-\}_2$ by an occurrence of $\{-,-\}_1$.
Since the linearized Jordan identity has degree 4, we obtain four identities of degree 4
relating the two new operations:
\allowdisplaybreaks
\begin{align}
&
\{ \{ \{ a, c \}_1, b \}_1, d \}_1
+
\{ \{ \{ a, d \}_1, b \}_1, c \}_1
+
\{ \{ \{ c, d \}_2, b \}_2, a \}_2
\notag
\\
&\quad
-
\{ \{ a, c \}_1, \{ b, d \}_1 \}_1
-
\{ \{ a, d \}_1, \{ b, c \}_1 \}_1
-
\{ \{ c, d \}_2, \{ b, a \}_2 \}_2
\equiv 0,
\label{qj1}
\\
&
\{ \{ \{ a, c \}_2, b \}_2, d \}_1
+
\{ \{ \{ a, d \}_2, b \}_2, c \}_1
+
\{ \{ \{ c, d \}_2, b \}_2, a \}_1
\notag
\\
&\quad
-
\{ \{ a, c \}_2, \{ b, d \}_1 \}_2
-
\{ \{ a, d \}_2, \{ b, c \}_1 \}_2
-
\{ \{ c, d \}_2, \{ b, a \}_1 \}_2
\equiv 0,
\label{qj2}
\\
&
\{ \{ \{ a, c \}_2, b \}_1, d \}_1
+
\{ \{ \{ a, d \}_2, b \}_2, c \}_2
+
\{ \{ \{ c, d \}_1, b \}_1, a \}_1
\notag
\\
&\quad
-
\{ \{ a, c \}_2, \{ b, d \}_1 \}_1
-
\{ \{ a, d \}_2, \{ b, c \}_2 \}_2
-
\{ \{ c, d \}_1, \{ b, a \}_1 \}_1
\equiv 0,
\label{qj3}
\\
&
\{ \{ \{ a, c \}_2, b \}_2, d \}_2
+
\{ \{ \{ a, d \}_2, b \}_1, c \}_1
+
\{ \{ \{ c, d \}_2, b \}_1, a \}_1
\notag
\\
&\quad
-
\{ \{ a, c \}_2, \{ b, d \}_2 \}_2
-
\{ \{ a, d \}_2, \{ b, c \}_1 \}_1
-
\{ \{ c, d \}_2, \{ b, a \}_1 \}_1
\equiv 0.
\label{qj4}
\end{align}
In identities \eqref{qj1}--\eqref{qj4}, we replace every instance of 
the second operation by the opposite of the first operation:
\allowdisplaybreaks
\begin{align}
&
\{ \{ \{ a, c \}_1, b \}_1, d \}_1
+
\{ \{ \{ a, d \}_1, b \}_1, c \}_1
+
\{ a, \{ b, \{ d, c \}_1 \}_1 \}_1
\notag
\\
&\quad
-
\{ \{ a, c \}_1, \{ b, d \}_1 \}_1
-
\{ \{ a, d \}_1, \{ b, c \}_1 \}_1
-
\{ \{ a, b \}_1, \{ d, c \}_1 \}_1
\equiv 0,
\label{qj1'}
\\
&
\{ \{ b, \{ c, a \}_1 \}_1, d \}_1
+
\{ \{ b, \{ d, a \}_1 \}_1, c \}_1
+
\{ \{ b, \{ d, c \}_1 \}_1, a \}_1
\notag
\\
&\quad
-
\{ \{ b, d \}_1, \{ c, a \}_1 \}_1
-
\{ \{ b, c \}_1, \{ d, a \}_1 \}_1
-
\{ \{ b, a \}_1, \{ d, c \}_1 \}_1
\equiv 0,
\label{qj2'}
\\
&
\{ \{ \{ c, a \}_1, b \}_1, d \}_1
+
\{ c, \{ b, \{ d, a \}_1 \}_1 \}_1
+
\{ \{ \{ c, d \}_1, b \}_1, a \}_1
\notag
\\
&\quad
-
\{ \{ c, a \}_1, \{ b, d \}_1 \}_1
-
\{ \{ c, b \}_1, \{ d, a \}_1 \}_1
-
\{ \{ c, d \}_1, \{ b, a \}_1 \}_1
\equiv 0,
\label{qj3'}
\\
&
\{ d, \{ b, \{ c, a \}_1 \}_1 \}_1
+
\{ \{ \{ d, a \}_1, b \}_1, c \}_1
+
\{ \{ \{ d, c \}_1, b \}_1, a \}_1
\notag
\\
&\quad
-
\{ \{ d, b \}_1, \{ c, a \}_1 \}_1
-
\{ \{ d, a \}_1, \{ b, c \}_1 \}_1
-
\{ \{ d, c \}_1, \{ b, a \}_1 \}_1
\equiv 0.
\label{qj4'}
\end{align}
Since we now have only one operation, we revert to a simpler notation, 
and write $\{a,b\}_1$ simply as $ab$.
Identities \eqref{qj1'}--\eqref{qj4'} take the following form:
\allowdisplaybreaks
\begin{align}
&
( ( a c ) b ) d + ( ( a d ) b ) c + a ( b ( d c ) ) - ( a c ) ( b d ) - 
( a d ) ( b c ) - ( a b ) ( d c ) \equiv 0,
\label{qj1''}
\\
&
( b ( c a ) ) d
+
( b ( d a ) ) c
+
( b ( d c ) ) a
-
( b d ) ( c a )
-
( b c ) ( d a )
-
( b a ) ( d c )
\equiv 0,
\label{qj2''}
\\
&
( ( c a ) b ) d
+
c ( b ( d a ) )
+
( ( c d ) b ) a
-
( c a ) ( b d )
-
( c b ) ( d a )
-
( c d ) ( b a )
\equiv 0,
\label{qj3''}
\\
&
d ( b ( c a ) )
+
( ( d a ) b ) c
+
( ( d c ) b ) a
-
( d b ) ( c a )
-
( d a ) ( b c )
-
( d c ) ( b a )
\equiv 0.
\label{qj4''}
\end{align}
Clearly \eqref{qj1''} becomes \eqref{qj3''} after the transposition $ac$,
and \eqref{qj1''} becomes \eqref{qj4''} after the cyclic permutation $adc$.
We discard \eqref{qj3''} and \eqref{qj4''} and retain \eqref{qj1''} and \eqref{qj2''}.

Part 2:
We include identities which state that in the first (resp.~second) argument 
of the second (resp.~first) new operation,
the two operations are interchangeable:
\[
\{ a, \{ b, c \}_1 \}_1 \equiv \{ a, \{ b, c \}_2 \}_1,
\qquad
\{ \{ a, b \}_1, c \}_2 \equiv \{ \{ a, b \}_2, c \}_2.
\]
Rewriting these in terms of the first operation gives
\[
\{ a, \{ b, c \}_1 \}_1 \equiv \{ a, \{ c, b \}_1 \}_1,
\qquad
\{ c, \{ a, b \}_1 \}_1 \equiv \{ c, \{ b, a \}_1 \}_1.
\]
These identities are both equivalent to right-commutativity $a(bc) \equiv a(cb)$.

Rearranging the terms in \eqref{qj1''} and applying right-commutativity gives
\[
( ( a c ) b ) d
-
( a c ) ( b d )
+
( ( a d ) b ) c
-
( a d ) ( b c )
-
( a b ) ( c d )
+
a ( b ( c d ) )
\equiv 0.
\]
This can be reformulated in terms of associators as follows:
\begin{equation}
( a c, b, d )
+
( a d, b, c )
-
( a, b, c d )
\equiv 0.
\label{qj1'''}
\end{equation}
If we assume characteristic not $2$, then identity \eqref{qj1'''} is 
equivalent to 
\begin{equation}
( a, b, c^2 ) \equiv 2 ( a c, b, c ).
\label{qj1''''}
\end{equation}
Apply right-commutativity to \eqref{qj2''} gives
\[
( b ( a c ) ) d
+
( b ( a d ) ) c
+
( b ( c d ) ) a
-
( b d ) ( a c )
-
( b c ) ( a d )
-
( b a ) ( c d )
\equiv 0.
\]
Setting $a = c = d$ (and dividing by 3) gives
\begin{equation}
( b a^2 ) a \equiv ( b a ) a^2.
\label{qj2''''}
\end{equation}
If we assume characteristic not $3$, then identities \eqref{qj2''} and \eqref{qj2''''} are equivalent.
\end{example}

\begin{definition} \label{JDdefinition}
Over a field of characteristic not $2, 3$, a \textbf{(right) Jordan dialgebra} 
is a vector space $D$ with a bilinear operation $D \times D \to D$, denoted $ab$, 
satisfying \textbf{right commutativity}, the \textbf{Osborn identity}, 
and the \textbf{right Jordan identity}:
  \[
  a(bc) \equiv a(cb),
  \qquad
  ( a, b, c^2 ) \equiv 2 ( a c, b, c ),
  \qquad
  ( b a^2 ) a \equiv ( b a ) a^2.
  \]
\end{definition}

\begin{remark}
The second identity in Definition \ref{JDdefinition} first appeared in Osborn \cite{Osborn}; 
see also Petersson \cite{Petersson}.
In a Jordan dialgebra, the Osborn identity is equivalent to
  \begin{equation} \label{osborn2}
  ( b, a^2, c ) = 2 ( b, a, c ) a.
  \end{equation}
The linearized forms of the Osborn identity, the right Jordan identity and \eqref{osborn2}
are
  \begin{align*}
  O_1(a,b,c,d) &= ((ba)c)d + ((bd)c)a - (b(ac))d - (b(cd))a - (b(ad))c + b((ad)c),
  \\
  RJ(a,b,c,d) &= (b(ac))d + (b(ad))c + (b(cd))a - (bd)(ac) - (bc)(ad) - (ba)(cd),
  \\
  O_2(a,b,c,d) &= ((ac)b)d + ((ad)b)c - (ab)(cd) - (ac)(bd) - (ad)(bc) + a((cd)b).
  \end{align*}
Using right commutativity it is easy to verify that
  \[
  O_2(a,b,c,d) = O_1(c,a,b,d) + RJ(b,a,c,d).
  \]
\end{remark}

\begin{remark}
The notion of a Jordan dialgebra was discovered independently by various authors during 
the last few years.
Kolesnikov \cite{Kolesnikov} introduced a functorial approach to varieties of associative 
and nonassociative dialgebras and obtained the defining identities for left Jordan dialgebras 
(the opposite identities to those of Definition \ref{JDdefinition}).
Vel\'asquez and Felipe \cite{VelasquezFelipe1} introduced the product $a \dashv b + b \vdash a$
in an associative dialgebra, showed that it satisfies the first and third identities of Definition 
\ref{JDdefinition}, and defined a quasi-Jordan algebra to be a vector space satisfying these
identities.
Bremner \cite{Bremner} used computer algebra to verify that $a \dashv b + b \vdash a$ also 
satisfies the second identity of Definition \ref{JDdefinition}.
\end{remark}


\section{Jordan triple disystems} \label{jordantripledisystems}

In this section we apply the KP algorithm to the defining identities of Jordan triple 
systems. We obtain a new variety of triple systems, which we call Jordan triple disystems,
with two trilinear operations. 

\begin{definition} \label{defjts}
A (linear) \textbf{Jordan triple system} (JTS) is a vector space $T$ over a field of characteristic not $2$ 
with a trilinear operation $T \times T \times T \to T$, denoted $\{-,-,-\}$,
satisfying these polynomial identities:
\begin{align*}
&
\{a,b,c\}
\equiv
\{c,b,a\},
\\
&
\{a,b,\{c,d,e\}\}
\equiv
\{\{a,b,c\},d,e\} - \{c,\{b,a,d\},e\} + \{c,d,\{a,b,e\}\}.
\end{align*}
\end{definition}

\begin{theorem} \label{KPidentities}
Applying the KP algorithm to Definition \ref{defjts} produces
two trilinear operations $\{-,-,-\}_1$ and $\{-,-,-\}_2$ satisfying these identities:
\begin{align*}
&
\{a,b,c\}_2 \equiv \{c,b,a\}_2,
\\
&
\{ a, \{ b, c, d \}_1, e \}_1 \equiv
\{ a, \{ b, c, d \}_2, e \}_1 \equiv
\{ a, \{ d, c, b \}_1, e \}_1,
\\
&
\{ a, b, \{ c, d, e \}_1 \}_1 \equiv
\{ a, b, \{ c, d, e \}_2 \}_1 \equiv
\{ a, b, \{ e, d, c \}_1 \}_1,
\\
&
\{ \{ a, b, c \}_1, d, e \}_2 \equiv
\{ \{ a, b, c \}_2, d, e \}_2 \equiv
\{ \{ c, b, a \}_1, d, e \}_2,
\\
&
\{ \{ e, d, c \}_1, b, a \}_1
\equiv
\{ \{ e, b, a \}_1, d, c \}_1
-
\{ e, \{ d, a, b \}_1, c \}_1
+
\{ e, d, \{ c, b, a \}_1 \}_1,
\\
&
\{ \{ e, d, c \}_2, b, a \}_1
\equiv
\{ \{ e, b, a \}_1, d, c \}_2
-
\{ e, \{ d, a, b \}_1, c \}_2
+
\{ e, d, \{ c, b, a \}_1 \}_2,
\\
&
\{ a, b, \{ c, d, e \}_1 \}_1
\equiv
\{ \{ a, b, c \}_1, d, e \}_1
-
\{ c, \{ b, a, d \}_2, e \}_2
+
\{ \{ a, b, e \}_1, d, c \}_1,
\\
&
\{ a, b, \{ c, d, e \}_1 \}_2
\equiv
\{ \{ a, b, c \}_2, d, e \}_1
-
\{ c, \{ b, a, d \}_1, e \}_2
+
\{ \{ a, b, e \}_2, d, c \}_1.
\end{align*}
\end{theorem}

\begin{proof}
Part 1:
First, we consider the identity of degree 3: $\{a,b,c\} - \{c,b,a\} \equiv 0$.
If we make $a$, $b$, $c$ in turn the central argument we obtain three identities:
\[
\{a,b,c\}_1 - \{c,b,a\}_3 \equiv 0,
\quad
\{a,b,c\}_2 - \{c,b,a\}_2 \equiv 0,
\quad
\{a,b,c\}_3 - \{c,b,a\}_1 \equiv 0.
\]
The first and third identities are both equivalent to $\{a,b,c\}_3 \equiv \{c,b,a\}_1$:
the third operation is the opposite of the first, and can be eliminated.
The second identity says that the second operation is symmetric in its 
first and third arguments: 
  \begin{equation} \label{op2symmetry}
  \{a,b,c\}_2 \equiv \{c,b,a\}_2.
  \end{equation}
Second, we consider the identity of degree 5,
\[
\{a,b,\{c,d,e\}\} - \{\{a,b,c\},d,e\} + \{c,\{b,a,d\},e\} - 
\{c,d,\{a,b,e\}\}
\equiv
0.
\]
If we make $a$, $b$, $c$, $d$, $e$ in turn the central argument we obtain five identities;
\allowdisplaybreaks
\begin{align*}
&
\{ a, b, \{ c, d, e \}_1 \}_1 -
\{ \{ a, b, c \}_1, d, e \}_1 +
\{ c, \{ b, a, d \}_2, e \}_2 -
\{ c, d, \{ a, b, e \}_1 \}_3
\equiv 0,
\\
&
\{ a, b, \{ c, d, e \}_1 \}_2 -
\{ \{ a, b, c \}_2, d, e \}_1 +
\{ c, \{ b, a, d \}_1, e \}_2 -
\{ c, d, \{ a, b, e \}_2 \}_3
\equiv 0,
\\
&
\{ a, b, \{ c, d, e \}_1 \}_3 -
\{ \{ a, b, c \}_3, d, e \}_1 +
\{ c, \{ b, a, d \}_1, e \}_1 -
\{ c, d, \{ a, b, e \}_1 \}_1
\equiv 0,
\\
&
\{ a, b, \{ c, d, e \}_2 \}_3 -
\{ \{ a, b, c \}_3, d, e \}_2 +
\{ c, \{ b, a, d \}_3, e \}_2 -
\{ c, d, \{ a, b, e \}_1 \}_2
\equiv 0,
\\
&
\{ a, b, \{ c, d, e \}_3 \}_3 -
\{ \{ a, b, c \}_3, d, e \}_3 +
\{ c, \{ b, a, d \}_3, e \}_3 -
\{ c, d, \{ a, b, e \}_3 \}_3
\equiv 0.
\end{align*}
We replace $\{a,b,c\}_3$ by the opposite of $\{a,b,c\}_1$; to save space we omit ``$\equiv 0$'':
\allowdisplaybreaks
\begin{align}
&
\{ a, b, \{ c, d, e \}_1 \}_1 -
\{ \{ a, b, c \}_1, d, e \}_1 +
\{ c, \{ b, a, d \}_2, e \}_2 -
\{ \{ a, b, e \}_1, d, c \}_1,
\label{id5-1}
\\
&
\{ a, b, \{ c, d, e \}_1 \}_2 -
\{ \{ a, b, c \}_2, d, e \}_1 +
\{ c, \{ b, a, d \}_1, e \}_2 -
\{ \{ a, b, e \}_2, d, c \}_1,
\label{id5-2}
\\
&
\{ \{ c, d, e \}_1, b, a \}_1 -
\{ \{ c, b, a \}_1, d, e \}_1 +
\{ c, \{ b, a, d \}_1, e \}_1 -
\{ c, d, \{ a, b, e \}_1 \}_1,
\label{id5-3}
\\
&
\{ \{ c, d, e \}_2, b, a \}_1 -
\{ \{ c, b, a \}_1, d, e \}_2 +
\{ c, \{ d, a, b \}_1, e \}_2 -
\{ c, d, \{ a, b, e \}_1 \}_2,
\label{id5-4}
\\
&
\{ \{ e, d, c \}_1, b, a \}_1 -
\{ e, d, \{ c, b, a \}_1 \}_1 +
\{ e, \{ d, a, b \}_1, c \}_1 -
\{ \{ e, b, a \}_1, d, c \}_1.
\label{id5-5}
\end{align}

Part 2:
We obtain the following 12 identities:
\allowdisplaybreaks
\begin{align*}
&
\{ a, \{ b, c, d \}_1, e \}_1 \equiv
\{ a, \{ b, c, d \}_2, e \}_1 \equiv
\{ a, \{ b, c, d \}_3, e \}_1,
\\
&
\{ a, b, \{ c, d, e \}_1 \}_1 \equiv
\{ a, b, \{ c, d, e \}_2 \}_1 \equiv
\{ a, b, \{ c, d, e \}_3 \}_1,
\\
&
\{ \{ a, b, c \}_1, d, e \}_2 \equiv
\{ \{ a, b, c \}_2, d, e \}_2 \equiv
\{ \{ a, b, c \}_3, d, e \}_2,
\\
&
\{ a, b, \{ c, d, e \}_1 \}_2 \equiv
\{ a, b, \{ c, d, e \}_2 \}_2 \equiv
\{ a, b, \{ c, d, e \}_3 \}_2,
\\
&
\{ \{ a, b, c \}_1, d, e \}_3 \equiv
\{ \{ a, b, c \}_2, d, e \}_3 \equiv
\{ \{ a, b, c \}_3, d, e \}_3,
\\
&
\{ a, \{ b, c, d \}_1, e \}_3 \equiv
\{ a, \{ b, c, d \}_2, e \}_3 \equiv
\{ a, \{ b, c, d \}_3, e \}_3.
\end{align*}
We replace $\{a,b,c\}_3$ by the opposite of $\{a,b,c\}_1$, obtaining
\allowdisplaybreaks
\begin{align}
&
\{ a, \{ b, c, d \}_1, e \}_1 \equiv
\{ a, \{ b, c, d \}_2, e \}_1 \equiv
\{ a, \{ d, c, b \}_1, e \}_1,
\label{bar12}
\\
&
\{ a, b, \{ c, d, e \}_1 \}_1 \equiv
\{ a, b, \{ c, d, e \}_2 \}_1 \equiv
\{ a, b, \{ e, d, c \}_1 \}_1,
\label{bar13}
\\
&
\{ \{ a, b, c \}_1, d, e \}_2 \equiv
\{ \{ a, b, c \}_2, d, e \}_2 \equiv
\{ \{ c, b, a \}_1, d, e \}_2,
\label{bar21}
\\
&
\{ a, b, \{ c, d, e \}_1 \}_2 \equiv
\{ a, b, \{ c, d, e \}_2 \}_2 \equiv
\{ a, b, \{ e, d, c \}_1 \}_2,
\label{bar23}
\end{align}
and other equivalent identities;
note that \eqref{bar23} follows from \eqref{bar21} by using \eqref{op2symmetry}.

We now see that \eqref{id5-3} becomes \eqref{id5-5} by the transposition $ce$
and using \eqref{bar12} and \eqref{bar13}.
We retain \eqref{id5-5}, which we write as a derivation property:
\begin{equation}
\{ \{ e, d, c \}_1, b, a \}_1
\equiv
\{ \{ e, b, a \}_1, d, c \}_1
-
\{ e, \{ d, a, b \}_1, c \}_1
+
\{ e, d, \{ c, b, a \}_1 \}_1.
\end{equation}
We also see that \eqref{id5-4} becomes another derivation property by using \eqref{bar23}:
\begin{equation}
\{ \{ c, d, e \}_2, b, a \}_1
\equiv
\{ \{ c, b, a \}_1, d, e \}_2
-
\{ c, \{ d, a, b \}_1, e \}_2
+
\{ c, d, \{ e, b, a \}_1 \}_2.
\end{equation}
This completes the proof.
\end{proof}

\begin{definition}
A \textbf{Jordan triple disystem} (JTD) is a vector space $D$ over a field of characteristic not $2$ 
with two trilinear operations 
$\{-,-,-\}_i \colon D \times D \times D \rightarrow D$ ($i = 1,2$)
satisfying the following identities:
\begin{align}
&
\{a,b,c\}_2 \equiv \{c,b,a\}_2,
\label{JTD1} \tag{J1}
\\
&
\{ a, \{ b, c, d \}_1, e \}_1 \equiv
\{ a, \{ b, c, d \}_2, e \}_1,
\label{JTD2} \tag{J2}
\\
&
\{ a, b, \{ c, d, e \}_1 \}_1 \equiv
\{ a, b, \{ c, d, e \}_2 \}_1,
\label{JTD3} \tag{J3}
\\
&
\{ \{ a, b, c \}_1, d, e \}_2 \equiv
\{ \{ a, b, c \}_2, d, e \}_2,
\label{JTD4} \tag{J4}
\\
&
\{ \{ e, d, c \}_1, b, a \}_1
\equiv
\{ \{ e, b, a \}_1, d, c \}_1
-
\{ e, \{ d, a, b \}_1, c \}_1
+
\{ e, d, \{ c, b, a \}_1 \}_1,
\label{JTD5} \tag{J5}
\\
&
\{ \{ e, d, c \}_2, b, a \}_1
\equiv
\{ \{ e, b, a \}_1, d, c \}_2
-
\{ e, \{ d, a, b \}_1, c \}_2
+
\{ e, d, \{ c, b, a \}_1 \}_2,
\label{JTD6} \tag{J6}
\\
&
\{ a, b, \{ c, d, e \}_1 \}_1
\equiv
\{ \{ a, b, c \}_1, d, e \}_1
-
\{ c, \{ b, a, d \}_2, e \}_2
+
\{ \{ a, b, e \}_1, d, c \}_1,
\label{JTD7} \tag{J7}
\\
&
\{ a, b, \{ c, d, e \}_1 \}_2
\equiv
\{ \{ a, b, c \}_2, d, e \}_1
-
\{ c, \{ b, a, d \}_1, e \}_2
+
\{ \{ a, b, e \}_2, d, c \}_1.
\label{JTD8} \tag{J8}
\end{align}
(We have omitted the redundant second identities in lines 2, 3 and 4.)
\end{definition}

We conclude this section with some examples of Jordan triple disystems.

\begin{example}
Let $T$ be a Jordan triple system with product $\{-,-,-\}$ over a field of characteristic not $2$.
It is straightforward to check that $T$ becomes a Jordan triple disystem by setting 
$\{-, -, -\}_1 = \{-,-,-\}_2 = \{-,-,-\}$. 
In particular, every associative algebra gives rise to a Jordan triple disystem by defining 
$\{a, b, c\}_1 = \{a, b, c\}_2 = abc + cba$. 
(For details see Section \ref{associativediproducts}.) 
\end{example}

\begin{example}
Let $A$ be a differential associative algebra in the sense of Loday \cite{LodaySurvey}:
that is, $A$ is an associative algebra with product $a \cdot b$ 
together with a linear map $d\colon A \to A$ such that $d^2 = 0$ and 
$d(a \cdot b) = d(a) \cdot b + a \cdot d(b)$ for all $a, b \in A$.
One endows $A$ with a dialgebra structure by defining 
$a \dashv b = a \cdot d(b)$ and $a \vdash b = d(a) \cdot b$. 
It follows from Section \ref{associativediproducts} that $A$ becomes a Jordan triple disystem by defining
  \[
  \{a,b,c\}_1 = a \cdot d(b) \cdot d(c) + d(c) \cdot d(b) \cdot a,
  \quad 
  \{a,b,c\}_2 = d(a) \cdot b \cdot d(c) + d(c) \cdot b \cdot d(a).
  \]
\end{example}

\begin{example}
Let $L$ be a Leibniz algebra over a field of characteristic not $2$. 
If an element $x \in L$ satisfies $[x,[x,[x,L]]] = \{0\}$ then we define
  \[
  L^{(x)} = \{ \, y \in L \mid [x,[x,y]] = 0 \, \}.
  \] 
By Vel\'asquez and Felipe \cite{VelasquezFelipe1} 
(see also Gubarev and Kolesnikov \cite{GubarevKolesnikov})
we know that the quotient space $L_x = L/L^{(x)}$ 
becomes a Jordan dialgebra with the product $ab = [a,[b,x]]$. 
It follows from Section \ref{Jordandialgebrasection} that
$L_x$ has the structure of a Jordan triple disystem with the trilinear operations
defined as follows for all $a, b, c \in L$:
  \begin{align*}
  &
  \{a, b, c\}_1  =  [[a, [b, x]], [c, x]] - [[a, [c, x]], [b, x]] + [a, [b, [c, x]]],
  \\
  &
  \{a, b, c\}_2  =  [[b, [a, x]], [c, x]] + [[b, [c, x]], [a, x]] - [b, [a, [c, x]]].
  \end{align*}
\end{example}


\section{Jordan triple diproducts in an associative dialgebra} \label{associativediproducts}

In this section, we study two trilinear operations in an associative dialgebra. 
We use computer algebra to determine the identities satisfied by these operations of degree $\le 5$
and prove that these identities are equivalent to those of Theorem \ref{KPidentities}.
We start by recalling the algorithm applied by Bremner and S\'anchez-Ortega \cite{BSO} 
to the alternating ternary sum. 
In the general case it converts a multilinear operation of degree $n$ in an associative algebra 
into a family of $n$ multilinear operations of degree $n$ in an associative dialgebra.

\begin{definition} \label{dioperations} {\bf BSO algorithm.} 

We start with a multilinear operation $\omega$ of degree $n$ in an 
associative algebra over a field $\mathbb{F}$, which we can identify with an element of the group algebra
$\mathbb{F} S_n$:
\[
\omega( a_1, a_2, \dots, a_n )
=
\sum_{\sigma \in S_n}
x_\sigma \, a_{\sigma(1)} a_{\sigma(2)} \cdots a_{\sigma(n)}
\qquad
(x_\sigma \in \mathbb{F}).
\]
For all $i, j = 1, 2, \dots, n$ we collect the terms in which $a_i$ is in position $j$;
we write $S_n^{j,i}$ for the set of permutations $\sigma$ with $\sigma(j) = i$:
\[
\omega_i( a_1, a_2, \dots, a_n )
=
\sum_{j=1}^n
\sum_{S_n^{j,i}}
x_\sigma \, a_{\sigma(1)} \cdots a_{\sigma(j-1)} a_i a_{\sigma(j+1)} \cdots a_{\sigma(n)}.
\]
We define $n$ new multilinear operations in an associative dialgebra;
$\widehat{\omega}_i$ is obtained from $\omega$ by making $a_i$ the center of each dialgebra monomial:
\[
\widehat{\omega}_i( a_1, a_2, \dots, a_n )
=
\sum_{j=1}^n
\sum_{S_n^{(i)}}
x_\sigma \, a_{\sigma(1)} \cdots a_{\sigma(j-1)} \widehat{a}_i 
a_{\sigma(j+1)} \cdots a_{\sigma(n)}.
\]
\end{definition}

\begin{definition} \label{defjtp}
The \textbf{Jordan triple product} in an associative algebra $A$ 
over a field of characteristic not $2$ is the trilinear operation
\[
(a,b,c) = abc + cba.
\]
\end{definition}

\begin{definition}
The \textbf{Jordan triple diproducts} are obtained by applying the BSO algorithm to the Jordan triple product:
\[
(a,b,c)_1 = \widehat{a} b c + c b \widehat{a},
\qquad
(a,b,c)_2 = a \widehat{b} c + c \widehat{b} a,
\qquad
(a,b,c)_3 = a b \widehat{c} + \widehat{c} b a.
\]
It is clear that $(a,b,c)_3 = (c,b,a)_1$, so we will only consider 
$(a,b,c)_1$ and $(a,b,c)_2$.
\end{definition}

In the rest of this section, 
we use computer algebra to determine the multilinear polynomial identities of degrees 3 and 5
satisfied by the Jordan triple diproducts $(\cdots)_1$ and $(\cdots)_2$ over a field of characteristic 0.

\subsection{Degree 3: operation 1}

In this case a polynomial identity is a linear combination of the six permutations of $(a,b,c)_1$:
\[
x_1 (a,b,c)_1 +
x_2 (a,c,b)_1 +
x_3 (b,a,c)_1 +
x_4 (b,c,a)_1 +
x_5 (c,a,b)_1 +
x_6 (c,b,a)_1.
\]
We expand each diproduct to obtain a linear combination of the 18 multilinear dialgebra monomials of degree 3 
ordered as follows:
\[
\widehat{a}bc, \widehat{a}cb, \widehat{b}ac, \widehat{b}ca, 
\widehat{c}ab, \widehat{c}ba,
a\widehat{b}c, a\widehat{c}b, b\widehat{a}c, b\widehat{c}a, 
c\widehat{a}b, c\widehat{b}a,
ab\widehat{c}, ac\widehat{b}, ba\widehat{c}, bc\widehat{a}, 
ca\widehat{b}, cb\widehat{a}.
\]
We construct the $18 \times 6$ matrix $E$ in which the $(i,j)$ entry is the coefficient of 
the $i$-th dialgebra monomial in the expansion of the $j$-th diproduct monomial:
\[
E^t
=
\left[ \tiny
\begin{array}{cccccccccccccccccc}
1 & . & . & . & . & . & . & . & . & . & . & . & . & . & . & . & . & 1 \\
. & 1 & . & . & . & . & . & . & . & . & . & . & . & . & . & 1 & . & . \\
. & . & 1 & . & . & . & . & . & . & . & . & . & . & . & . & . & 1 & . \\
. & . & . & 1 & . & . & . & . & . & . & . & . & . & 1 & . & . & . & . \\
. & . & . & . & 1 & . & . & . & . & . & . & . & . & . & 1 & . & . & . \\
. & . & . & . & . & 1 & . & . & . & . & . & . & 1 & . & . & . & . & .
\end{array}
\right]
\]
The coefficient vectors of the polynomial identities satisfied by $(-,-,-)_1$ 
are the vectors in the nullspace of $E$, which is the zero subspace since $\mathrm{rank}(E) = 6$.

\begin{lemma}
The diproduct $(\cdots)_1$ satisfies no polynomial identity of degree 3.
\end{lemma}

\subsection{Degree 3: operation 2}

Replacing $(\cdots)_1$ by $(\cdots)_2$ gives the matrix
\[
E^t
=
\left[ \tiny
\begin{array}{cccccccccccccccccc}
. & . & . & . & . & . & 1 & . & . & . & . & 1 & . & . & . & . & . & . \\
. & . & . & . & . & . & . & 1 & . & 1 & . & . & . & . & . & . & . & . \\
. & . & . & . & . & . & . & . & 1 & . & 1 & . & . & . & . & . & . & . \\
. & . & . & . & . & . & . & 1 & . & 1 & . & . & . & . & . & . & . & . \\
. & . & . & . & . & . & . & . & 1 & . & 1 & . & . & . & . & . & . & . \\
. & . & . & . & . & . & 1 & . & . & . & . & 1 & . & . & . & . & . & .
\end{array}
\right]
\]
We compute the row canonical form and obtain the canonical basis of the nullspace:
\[
\left[\begin{array}{cccccc} 0 & -1 & 0 & 1 & 0 & 0 \end{array} \right],
\quad
\left[\begin{array}{cccccc} 0 & 0 & -1 & 0 & 1 & 0 \end{array} \right],
\quad
\left[\begin{array}{cccccc} -1 & 0 & 0 & 0 & 0 & 1 \end{array} \right].
\]

\begin{lemma} \label{operation2degree3}
Every polynomial identity of degree 3 satisfied by the diproduct $(\cdots)_2$ 
follows from the symmetry in the first and third arguments: $(a,b,c)_2 \equiv (c,b,a)_2$.
\end{lemma}

\subsection{Degree 5: operation 1}

Since $(\cdots)_1$ satisfies no polynomial identity of degree 3, 
we consider three association types of degree 5:
\[
((a,b,c)_1,d,e)_1,
\qquad
(a,(b,c,d)_1,e)_1,
\qquad
(a,b,(c,d,e)_1)_1.
\]
There are $3 \cdot 5!$ multilinear diproduct monomials of these three types,
and $5 \cdot 5!$ dialgebra monomials of the forms
$\widehat{a}bcde$,
$a\widehat{b}cde$,
$ab\widehat{c}de$,
$abc\widehat{d}e$,
$abcd\widehat{e}$.
We expand the diproduct monomials in an associative dialgebra and obtain
\begin{align*}
((a,b,c)_1,d,e)_1
&=
\widehat{a} b c d e + c b \widehat{a} d e + e d \widehat{a} b c + e d c 
b \widehat{a},
\\
(a,(b,c,d)_1,e)_1
&=
\widehat{a} b c d e + \widehat{a} d c b e + e b c d \widehat{a} + e d c 
b \widehat{a},
\\
(a,b,(c,d,e)_1)_1
&=
\widehat{a} b c d e + \widehat{a} b e d c + c d e b \widehat{a} + e d c 
b \widehat{a}.
\end{align*}
We construct the $600 \times 360$ matrix $E$ in which the $(i,j)$ entry is 
the coefficient of the $i$-th dialgebra monomial in the expansion of the $j$-th diproduct monomial.
Using a computer algebra system, we find that $\mathrm{rank}(E) = 150$,
and hence the nullspace has dimension 210.
We compute the canonical basis of the nullspace, and find that all the components are $\pm 1$.
We sort these vectors by increasing number of nonzero components:
there are 30 with two, 120 with four, and 60 with six.
We construct the $480 \times 360$ matrix $M$
with a $360 \times 360$ upper block and a $120 \times 360$ lower block.
For each nullspace basis vector, we perform the following computations:
\begin{enumerate}
\item
Apply all permutations of $a,b,c,d,e$ to the corresponding linear combination of diproduct monomials, 
and store the results in the lower block.
\item
Compute the row canonical form; the lower block is now zero, and
the upper block contains a basis for the subspace of the nullspace generated by
the nullspace basis vectors up to the current vector.
\item
If the rank of the matrix has increased from the previous vector to the 
current vector,
then we record the current vector as a generator.
\end{enumerate}
We obtain three generators of the nullspace, corresponding to these identities:
\begin{align*}
&
( a, (b,c,d)_1, e )_1 - ( a, (d,c,b)_1, e )_1
\equiv 0,
\\
&
( (e,b,a)_1, d, c )_1
- ( (e,d,c)_1, b, a )_1
- ( e, (b,a,d)_1, c )_1
+ ( e, d, (c,b,a)_1 )_1
\equiv 0,
\\
&
( (e,b,a)_1, d, c )_1
- ( (e,d,c)_1, b, a )_1
- ( e, (b,a,d)_1, c )_1
+ ( e, d, (a,b,c)_1 )_1
\equiv 0.
\end{align*}
A similar computation shows that no two of these identities generate the entire nullspace.
The difference of the second and third identities is
\[
( e, d, (c,b,a)_1 )_1 - ( e, d, (a,b,c)_1 )_1
\equiv 0,
\]
and this gives a simpler set of three generating identities.

\begin{proposition} \label{operation1degree5}
Every polynomial identity of degree 5 satisfied by the diproduct $(\cdots)_1$ 
is a consequence of these three independent identities:
\begin{align*}
&
( a, (b,c,d)_1, e )_1 \equiv ( a, (d,c,b)_1, e )_1,
\qquad
( a, b, (c,d,e)_1 )_1 \equiv ( a, b, (e,d,c)_1 )_1,
\\
&
( (e,d,c)_1, b, a )_1 \equiv
( (e,b,a)_1, d, c )_1 - ( e, (d,a,b)_1, c )_1 + ( e, d, (c,b,a)_1 )_1.
\end{align*}

\end{proposition}

\subsection{Degree 5: operation 2}

Lemma \ref{operation2degree3} implies that we need to consider only two association types 
for the diproduct $(\cdots)_2$ of degree 5:
\begin{align*}
((a,b,c)_2,d,e)_2
&=
a b c \widehat{d} e + c b a \widehat{d} e + e \widehat{d} a b c + e 
\widehat{d} c b a,
\\
(a,(b,c,d)_2,e)_2
&=
a b \widehat{c} d e + a d \widehat{c} b e + e b \widehat{c} d a + e d 
\widehat{c} b a.
\end{align*}
The first type is symmetric in $a$ and $c$, giving $5!/2$ monomials;
the second is symmetric in $a$ and $e$ and in $b$ and $d$, giving $5!/4$ monomials.
We construct the $600 \times 90$ matrix $E$ in which the $(i,j)$ entry is 
the coefficient of the $i$-th dialgebra monomial in the expansion of the $j$-th diproduct monomial.
We find that $\mathrm{rank}(E) = 90$.

\begin{proposition}
Every polynomial identity of degree 5 satisfied by the diproduct $(\cdots)_2$ is a consequence
of the symmetry of Lemma \ref{operation2degree3}.
\end{proposition}

\subsection{Degree 5: operations 1 and 2}

We now consider multilinear polynomial identities which involve both diproducts $(\cdots)_1$ and $(\cdots)_2$.
In addition to the three association types for $(\cdots)_1$ and the two association types for $(\cdots)_2$,
we must also consider the five association types involving both operations:
\begin{center}
\begin{tabular}{rlrl}
1: &\quad $((a,b,c)_1,d,e)_1$ &\qquad
6: &\quad $((a,b,c)_2,d,e)_1$ \\
2: &\quad $(a,(b,c,d)_1,e)_1$ &\qquad
7: &\quad $(a,(b,c,d)_2,e)_1$ \\
3: &\quad $(a,b,(c,d,e)_1)_1$ &\qquad
8: &\quad $(a,b,(c,d,e)_2)_1$ \\
4: &\quad $((a,b,c)_2,d,e)_2$ &\qquad
9: &\quad $((a,b,c)_1,d,e)_2$ \\
5: &\quad $(a,(b,c,d)_2,e)_2$ &\qquad
10: &\quad $(a,(b,c,d)_1,e)_2$
\end{tabular}
\end{center}
Lemma \ref{operation2degree3} and Proposition \ref{operation1degree5} imply the following results:
\begin{center}
\begin{tabular}{rllrll}
&\quad symmetries &\quad monomials & &\quad symmetries &\quad monomials \\
1: &\quad $-$ &\quad 5! = 120 &\qquad
6: &\quad $a \leftrightarrow c$ &\quad 5!/2 = 60 \\
2: &\quad $b \leftrightarrow d$ &\quad 5!/2 = 60 &\qquad
7: &\quad $b \leftrightarrow d$ &\quad 5!/2 = 60 \\
3: &\quad $c \leftrightarrow e$ &\quad 5!/2 = 60 &\qquad
8: &\quad $c \leftrightarrow e$ &\quad 5!/2 = 60 \\
4: &\quad $a \leftrightarrow c$ &\quad 5!/2 = 60 &\qquad
9: &\quad $-$ &\quad 5! = 120 \\
5: &\quad $a \leftrightarrow e$; $b \leftrightarrow d$ &\quad 5!/4 = 30 &\qquad
10: &\quad $a \leftrightarrow e$ &\quad 5!/2 = 60
\end{tabular}
\end{center}
The total number of multilinear diproduct monomials is 690.
In this case, the matrix $E$ has size $600 \times 690$,
and as before the $(i,j)$ entry is the coefficient of the $i$-th dialgebra monomial
in the expansion of the $j$-th diproduct monomial.
For such a large matrix we use modular arithmetic to do Gaussian elimination
in order to control the memory requirements.
Since we consider multilinear monomials, all vector spaces are representations of $S_5$;
if use a modulus $p > 5$ we will obtain
dimensions equivalent to those which we would have obtained using rational arithmetic.
We find that $\mathrm{rank}(E) = 250$, and hence the nullspace has dimension 440.
We compute the canonical basis and sort it by increasing number of nonzero components.

We now construct an $810 \times 690$ matrix $M$ with a $690 \times 690$ upper block and a $120 \times 690$ lower block.
Before processing the 440 nullspace identities,
we first process the third identity of Proposition \ref{operation1degree5}, which we rewrite as follows:
\[
( (a,b,c)_1, d, e )_1 - ( (a,d,e)_1, b, c )_1 + ( a, (b,e,d)_1, c )_1 - 
( a, b, (c,d,e)_1 )_1
\equiv
0.
\]
This is the only known identity that has not already been taken into account; 
the other identities from Lemma \ref{operation2degree3} and Proposition \ref{operation1degree5} 
express symmetries which have been assumed in our enumeration of the diproduct monomials.
This identity generates a 90-dimensional subspace of the nullspace.
Continuing with the 440 nullspace vectors, we obtain six additional generators (we omit ``$\equiv 0$''):
\allowdisplaybreaks
\begin{align}
&
( (c,b,a)_2, d, e )_2 - ( (a,b,c)_1, d, e )_2,
\label{caid1}
\\
&
( a, (b,c,d)_1, e )_1 - ( a, (b,c,d)_2, e )_1,
\label{caid2}
\\
&
( (e,d,c)_2, b, a )_2 - ( (a,b,e)_2, d, c )_1 - ( (a,b,c)_2, d, e )_1 + 
( e, (b,a,d)_1, c )_2,
\label{caid3}
\\
&
( (e,b,a)_2, d, c )_2 + ( (c,b,a)_2, d, e )_2 - ( (e,d,c)_2, b, a )_1 - 
( e, (d,a,b)_1, c )_2,
\label{caid4}
\\
&
( (e,b,a)_1, d, c )_1 - ( (e,d,c)_1, b, a )_1 - ( e, (b,a,d)_1, c )_1 + 
( e, d, (a,b,c)_2 )_1,
\label{caid5}
\\
&
( (a,b,c)_1, d, e )_1 + ( (a,d,c)_1, b, e )_1 - ( a, (b,c,d)_1, e )_1 - 
( c, (b,a,d)_2, e )_2.
\label{caid6}
\end{align}
Further calculations show that if we take any proper subset of these six identities, 
and combine it with the third identity of Proposition \ref{operation1degree5}, 
then the resulting set of identities does not generate the entire nullspace. 
It follows that identities \eqref{caid1}--\eqref{caid6} are independent,
modulo the third identity of Proposition \ref{operation1degree5}.
Using Lemma \ref{operation2degree3} and Proposition \ref{operation1degree5}, 
we see that \eqref{caid1}, \eqref{caid2}, \eqref{caid5} are equivalent to
\begin{align*}
&
( (a,b,c)_2, d, e )_2 \equiv ( (a,b,c)_1, d, e )_2,
\\
&
( a, (b,c,d)_1, e )_1 \equiv ( a, (b,c,d)_2, e )_1,
\\
&
( (e,b,a)_1, d, c )_1 - ( (e,d,c)_1, b, a )_1 - ( e, (d,a,b)_1, c )_1 + 
( e, d, (c,b,a)_2 )_1
\equiv 0.
\end{align*}
From the last identity we subtract the third identity of Proposition \ref{operation1degree5} and obtain
\[
( a, b, (c,d,e)_1 )_1 \equiv ( a, b, (c,d,e)_2 )_1.
\]
We now see that identities \eqref{caid1}--\eqref{caid6} are equivalent 
to the following six identities (we omit ``$\equiv 0$'' in the last three):
\allowdisplaybreaks
\begin{align*}
&
( (a,b,c)_2, d, e )_2 \equiv ( (a,b,c)_1, d, e )_2,
\\
&
( a, (b,c,d)_1, e )_1 \equiv ( a, (b,c,d)_2, e )_1,
\\
&
( a, b, (c,d,e)_1 )_1 \equiv ( a, b, (c,d,e)_2 )_1,
\\
&
( (e,d,c)_2, b, a )_2 - ( (a,b,e)_2, d, c )_1 - ( (a,b,c)_2, d, e )_1 + 
( e, (b,a,d)_1, c )_2,
\\
&
( (e,b,a)_2, d, c )_2 + ( (c,b,a)_2, d, e )_2 - ( (e,d,c)_2, b, a )_1 - 
( e, (d,a,b)_1, c )_2,
\\
&
( (a,b,c)_1, d, e )_1 + ( (a,d,c)_1, b, e )_1 - ( a, (b,c,d)_1, e )_1 - 
( c, (b,a,d)_2, e )_2.
\end{align*}
The last three identities are equivalent (assuming known symmetries) to
\begin{align*}
&
( (e,d,c)_1, b, a )_2
\equiv
( (e,b,a)_2, d, c )_1 - ( e, (b,a,d)_1, c )_2 + ( (c,b,a)_2, d, e )_1.
\\
&
( (e,d,c)_2, b, a )_1
\equiv
( (e,b,a)_1, d, c )_2 - ( e, (d,a,b)_1, c )_2 + ( e, d, (c,b,a)_1 )_2.
\\
&
( a, (b,c,d)_2, e )_2
\equiv
( (c,b,a)_1, d, e )_1 + ( (c,d,a)_1, b, e )_1 - ( c, (b,a,d)_1, e )_1.
\end{align*}
We have proved the following result.

\begin{theorem} \label{identities Maple}
Every polynomial identity of degree $\le 5$ satisfied by the Jordan triple diproducts
$(\cdots)_1$ and $(\cdots)_2$ in an associative dialgebra
is a consequence of these eight independent identities:
\begin{align*}
&
(a,b,c)_2 \equiv (c,b,a)_2,
\\
&
( a, (b,c,d)_1, e )_1 \equiv ( a, (b,c,d)_2, e )_1,
\\
&
( a, b, (c,d,e)_1 )_1 \equiv ( a, b, (c,d,e)_2 )_1,
\\
&
( (a,b,c)_1, d, e )_2 \equiv ( (a,b,c)_2, d, e )_2,
\\
&
( (e,d,c)_1, b, a )_1
\equiv
( (e,b,a)_1, d, c )_1 - ( e, (d,a,b)_1, c )_1 + ( e, d, (c,b,a)_1 )_1,
\\
&
( (e,d,c)_2, b, a )_1
\equiv
( (e,b,a)_1, d, c )_2 - ( e, (d,a,b)_1, c )_2 + ( e, d, (c,b,a)_1 )_2,
\\
&
( (e,d,c)_1, b, a )_2
\equiv
( (e,b,a)_2, d, c )_1 - ( e, (b,a,d)_1, c )_2 + ( (c,b,a)_2, d, e )_1,
\\
&
( a, (b,c,d)_2, e )_2
\equiv
( (c,b,a)_1, d, e )_1 + ( (c,d,a)_1, b, e )_1 - ( c, (b,a,d)_1, e )_1.
\end{align*}
\end{theorem}

\begin{remark}
These identities are equivalent to those of Theorem \ref{KPidentities}. 
We have
\begin{align*}
&
(a, (b,c,d)_1, e)_1 
\equiv 
(a, (b, c, d)_2, e)_1 
\equiv 
(a, (d, c, b)_2, e)_1
\equiv
(a, (d, c, b)_1, e)_1,
\\
&
(a, b,(c,d,e)_1)_1 
\equiv 
(a, b, (c,d, e)_2)_1 
\equiv 
(a, b, (e,d,c)_2)_1
\equiv
(a, b, (e, d, c)_1)_1,
\\
&
((a,b,c)_1, d , e)_2 
\equiv 
((a,b,c)_2, d, e)_2 
\equiv 
((c,b,a)_2, d, e)_2
\equiv
((c,b,a)_1, d, e)_2,
\end{align*} 
corresponding to the identities in the first four lines displayed in Theorem \ref{KPidentities}. 
The fifth and sixth identities of Theorem \ref{identities Maple} coincide with the fifth and sixth lines 
displayed in Theorem \ref{KPidentities}.
Using the known identities,
\begin{align*}
&
( (e,d,c)_1, b, a )_2
\equiv 
( (c,d,e)_1, b, a )_2
\equiv 
( a, b, (c,d,e)_1 )_2,
\\
&
( (e,b,a)_2, d, c )_1
\equiv
( (a,b,e)_2, d, c )_1,
\qquad
( (c,b,a)_2, d, e )_1
\equiv
( (a,b,c)_2, d, e )_1,
\\
&
( e,(b,a, d)_1, c )_2
\equiv
( c, (b,a, d)_1, e )_2,
\end{align*} 
we see that the seventh identity of Theorem \ref{identities Maple} becomes the eighth line 
displayed in Theorem \ref{KPidentities}. 
To conclude, we apply the transposition $ac$ to the eighth identity of Theorem \ref{identities Maple}:
  \[
  ( c, (b,a, d)_2, e )_2  
  \equiv 
  ( (a,b,c)_1, d, e )_1 + ( (a,d,c)_1, b, e )_1 - ( a, (b,c, d)_1, e )_1 
  \]
Using this equivalent form of the fifth identity of Theorem \ref{identities Maple},
  \[
  ( (a,d,c)_1, b, e )_1 - ( a, (b,c, d)_1, e )_1 
  \equiv
  - ( a,b,(c, d, e)_1)_1 + ( (a,b,e)_1, d, c )_1,
  \]
we obtain
  \[
  ( c, (b,a, d)_2, e )_2  
  \equiv 
  ( (a,b,c)_1, d, e )_1 - ( a,b,(c, d, e)_1)_1 + ( (a,b,e)_1, d, c )_1, 
  \]
which is equivalent to the seventh line displayed in Theorem \ref{KPidentities}.
\end{remark}

\begin{theorem} \label{associativedialgebratheorem}
If $D$ is a subspace of an associative dialgebra over a field of characteristic not $2, 3, 5$
which is closed under the Jordan diproducts $(\cdots)_1$ and $(\cdots)_2$,
then $D$ is a Jordan triple disystem with respect to these operations.
\end{theorem}

We exclude characteristics 2, 3, 5 since our computer algebra methods require 
modular arithmetic with a prime greater than the degree of the identities.


\section{Conjecture: dialgebra operations and identities} \label{conjecturesection}

The results of Sections \ref{jordantripledisystems} and \ref{associativediproducts}
suggest a general conjecture relating multilinear operations in associative dialgebras
and their polynomial identities.

Fix a coefficient field $\mathbb{F}$ and an integer $n \ge 2$.
Let $\omega$ be a multilinear $n$-ary operation in an associative algebra over $\mathbb{F}$;
we can identify $\omega$ with an element of the group algebra $\mathbb{F} S_n$.
Fix a degree $d$ and consider the multilinear polynomial identities of degree $e \le d$
satisfied by $\omega$; we may assume that $d \equiv 1$ (mod $n{-}1$).
To be precise, let $A_e$ be the multilinear subspace of degree $e$ in
the free nonassociative $n$-ary algebra over $\mathbb{F}$ on $e$ generators.
Let $I_e$ be the subspace of $A_e$ consisting of those nonassociative $n$-ary polynomials
which vanish identically when the $n$-ary operation is replaced by $\omega$; 
thus $I_e$ is the kernel of the expansion map from $A_e$ to the group algebra $\mathbb{F} S_e$
regarded as the multilinear subspace of degree $e$ in the free associative algebra on $e$ generators.
By definition, the multilinear polynomial identities of degree $\le d$ satisfied by $\omega$ are
the direct sum
  \[
  I_d(\omega) = \bigoplus_{1 \le e \le d} I_e.
  \]
We apply the KP algorithm to the identities in $I_d(\omega)$ and obtain a set of multilinear
polynomial identities involving $n$ new $n$-ary operations.
To be precise, let $B_e$ be the multilinear subspace of degree $e$ in the free 
nonassociative multioperator algebra with $n$ operations of arity $n$.
Let KP$(I_e)$ be the subspace of $B_e$ obtained by applying the KP algorithm to the identities
in $I_e$, and define
  \[
  \mathrm{KP}_d(\omega) = \bigoplus_{1 \le e \le d} \mathrm{KP}(I_e).
  \]
We now consider a different path to the same destination.
We apply the BSO algorithm to $\omega$ and obtain $n$ multilinear operations
$\omega_1, \dots, \omega_n$ of arity $n$ in an associative dialgebra.
Consider the multilinear polynomial identities of degree $e \le d$ satisfied by
$\omega_1, \dots, \omega_n$.
To be precise, let $J_e$ be the subspace of $B_e$ consisting of those nonassociative 
polynomials in the $n$ operations which vanish identically when the $n$ operations 
are replaced by $\omega_1, \dots, \omega_n$; 
thus $J_e$ is the kernel of the expansion map from $B_e$ to the direct sum of $e$ copies of
the group algebra $\mathbb{F} S_e$
regarded as the multilinear subspace of degree $e$ in the free associative dialgebra on $e$ generators.
We define
  \[
  J_d(\omega_1,\dots,\omega_n) = \bigoplus_{1 \le e \le d} J_e.
  \]

\begin{conjecture} \label{conjecturestatement}
If the field $\mathbb{F}$ has characteristic 0 or $p > d$ then
  \[
  \mathrm{KP}_d(\omega) = J_d(\omega_1,\dots,\omega_n).
  \]
\end{conjecture}

Expressed less formally, the conjecture says that the following two processes produce the same results when
the group algebra $\mathbb{F} S_d$ is semisimple:
  \begin{itemize}
  \item
  Find the identities satisfied by $\omega$, and apply the KP algorithm.
  \item
  Apply the BSO algorithm, and find the identities satisfied by $\omega_1,\dots,\omega_n$.
  \end{itemize}
The conjecture is equivalent to the commutativity of the diagram in Figure \ref{conjecturediagram}.

\begin{figure}
\[
\begin{CD} 
\omega @>\text{BSO}>> \omega_1,\dots,\omega_n 
\\ 
@VVV @VVV 
\\ 
I_d(\omega) @>\text{KP}>> 
\begin{array}{c}
J_d(\omega_1,\dots,\omega_n)
\\
\stackrel{?}{=} \; \mathrm{KP}_d(\omega) 
\\
\end{array}
\end{CD}
\]
\caption{Diagrammatic formulation of Conjecture \ref{conjecturestatement}}
\label{conjecturediagram}
\end{figure}


\section{Jordan triple diproducts in a Jordan dialgebra} \label{Jordandialgebrasection}

Every Jordan algebra $J$ with operation $a \circ b$ gives rise to a Jordan triple system 
by considering the same underlying vector space with respect to the triple product 
  \begin{equation} \label{triple Jordan product}
  \langle a,b,c \rangle 
  = 
  (a \circ b) \circ c - (a \circ c) \circ b + a \circ (b \circ c). 
  \end{equation}
In this section we consider a Jordan dialgebra $D$ with operation $ab$ 
and the two trilinear operations which give rise to 
the structure of a Jordan triple disystem on the same underlying vector space.

The first trilinear operation $\langle \cdots \rangle_1$ can be obtained 
by replacing the Jordan algebra product $a \circ b$ in \eqref{triple Jordan product}
by the the Jordan dialgebra product $ab$ in $D$: 
  \begin{equation} \label{diop1}
  \langle a,b,c \rangle_1 = (ab)c - (ac)b + a(bc).
  \end{equation}
The second trilinear operation $\langle \cdots \rangle_2$ can be obtained from the first by 
transposing $a$ and $b$ and changing the signs of the second and third terms:
  \begin{equation} \label{diop2}
  \langle a,b,c \rangle_2 = (ba)c + (bc)a - b(ac).
  \end{equation}
It is straightforward to verify that in a special Jordan dialgebra, 
where the product is $ab = a \dashv b + b \vdash a$,
these operations reduce to twice the first and second Jordan diproducts in an associative dialgebra:
  \begin{align*}
  &\quad
  ( a \dashv b + b \vdash a ) \dashv c 
  + c \vdash ( a \dashv b + b \vdash a )
  \\
  &\quad
  - ( a \dashv c + c \vdash a ) \dashv b 
  - b \vdash ( a \dashv c + c \vdash a )
  \\
  &\quad
  + a \dashv ( b \dashv c + c \vdash b ) 
  + ( b \dashv c + c \vdash b ) \vdash a
  \\
  &=
  \widehat{a}bc + b\widehat{a}c + c\widehat{a}b + cb\widehat{a}
  - \widehat{a}cb - c\widehat{a}b - b\widehat{a}c - bc\widehat{a}
  + \widehat{a}bc + \widehat{a}cb + bc\widehat{a} + cb\widehat{a}
  \\
  &=
  2 \big( \widehat{a}bc + cb\widehat{a} \big),
  \\
  &\quad
  ( b \dashv a + a \vdash b ) \dashv c 
  + c \vdash ( b \dashv a + a \vdash b )
  \\
  &\quad
  + ( b \dashv c + c \vdash b ) \dashv a 
  + a \vdash ( b \dashv c + c \vdash b )
  \\
  &\quad
  - b \dashv ( a \dashv c + c \vdash a ) 
  - ( a \dashv c + c \vdash a ) \vdash b
  \\
  &=
  \widehat{b}ac + a\widehat{b}c + c\widehat{b}a + ca\widehat{b}
  + \widehat{b}ca + c\widehat{b}a + a\widehat{b}c + ac\widehat{b}
  - \widehat{b}ac - \widehat{b}ca - ac\widehat{b} - ca\widehat{b}
  \\
  &=
  2 \big( a\widehat{b}c + c\widehat{b}a \big).
  \end{align*}

\begin{lemma} \label{degree3-Jordan dialgebra}
In a Jordan dialgebra $D$ with operation $ab$, 
every polynomial identity of degree 3 satisfied by 
$\langle a,b,c \rangle_1$ and $\langle a,b,c \rangle_2$
is a consequence of the symmetry of $\langle a,b,c \rangle_2$ in its first and third arguments: 
$\langle a,b,c \rangle_2 \equiv \langle c,b,a \rangle_2$.
\end{lemma}

\begin{proof}
We use computer algebra.  We construct an $18 \times 24$ matrix $E$ in which
  \begin{itemize}
  \item the upper left $6 \times 12$ block contains the right commutative identities
  \item the lower left $12 \times 12$ block contains the expansions of 
  $\langle \cdots \rangle_1$ and $\langle \cdots \rangle_2$
  \item the lower right $12 \times 12$ block contains the identity matrix
  \item the upper right $6 \times 12$ block contains the zero matrix
  \end{itemize}
More precisely, columns 1--12 of the matrix correspond to the 12 multilinear monomials of degree 3
in the free nonassociative algebra,
  \[
  (ab)c, \, (ac)b, \, (ba)c, \, (bc)a, \, (ca)b, \, (cb)a, \,
  a(bc), \, a(cb), \, b(ac), \, b(ca), \, c(ab), \, c(ba),
  \]
and columns 13--24 correspond to the 12 multilinear monomials of degree 3 in the trilinear operations
$\langle \cdots \rangle_1$ and $\langle \cdots \rangle_2$,
  \begin{align*}
  &
  \langle a,b,c \rangle_1, \, 
  \langle a,c,b \rangle_1, \, 
  \langle b,a,c \rangle_1, \, 
  \langle b,c,a \rangle_1, \, 
  \langle c,a,b \rangle_1, \, 
  \langle c,b,a \rangle_1,
  \\
  &
  \langle a,b,c \rangle_1, \, 
  \langle a,c,b \rangle_1, \, 
  \langle b,a,c \rangle_1, \, 
  \langle b,c,a \rangle_1, \, 
  \langle c,a,b \rangle_1, \, 
  \langle c,b,a \rangle_1.
  \end{align*}
There are six permutations of the right-commutative identity:
  \[
  a(bc) {-} a(cb), \,
  a(cb) {-} a(bc), \,
  b(ac) {-} b(ca), \,
  b(ca) {-} b(ac), \,
  c(ab) {-} c(ba), \,
  c(ba) {-} c(ab).
  \]
Entry $(i,j)$ of the upper left block is the coefficient of the $j$-th nonassociative monomial
in the $i$-th permutation of the right-commutative identity.
Entry $(12{+}i,j)$ of the lower left block contains is the coefficient of the $j$-th nonassociative monomial
in the expansion of the $i$-th trilinear monomial using equations \eqref{diop1} and \eqref{diop2}.
This matrix is displayed in Figure \ref{degree3matrix}, using the symbols $\cdot, +, -$ for $0, 1, -1$.
The rank is 15, and the row canonical form is displayed in Figure \ref{degree3matrixrcf},
using the symbol $\ast$ for $\tfrac12$.
The dividing line between the upper and lower parts of the matrix lies immediately above row 13,
since that is the uppermost row whose leading 1 is in the right part of the matrix.
These rows represent the dependence relations among the expansions of the trilinear monomials
which hold as a result of the right-commutative identities.
The rows of the lower right $3 \times 12$ block of the row canonical form represent the 
permutations of the identity
$\langle a,b,c \rangle_2 - \langle c,b,a \rangle_2 \equiv 0$.
\end{proof}

  \begin{figure} \tiny
  \[
  \left[
  \begin{array}{cccccccccccc|cccccccccccc}
  . &\nn . &\nn . &\nn . &\nn . &\nn . &\nn + &\nn - &\nn . &\nn . &\nn . &\nn . &
  . &\nn . &\nn . &\nn . &\nn . &\nn . &\nn . &\nn . &\nn . &\nn . &\nn . &\nn . \\
  . &\nn . &\nn . &\nn . &\nn . &\nn . &\nn - &\nn + &\nn . &\nn . &\nn . &\nn . &
  . &\nn . &\nn . &\nn . &\nn . &\nn . &\nn . &\nn . &\nn . &\nn . &\nn . &\nn . \\
  . &\nn . &\nn . &\nn . &\nn . &\nn . &\nn . &\nn . &\nn + &\nn - &\nn . &\nn . &
  . &\nn . &\nn . &\nn . &\nn . &\nn . &\nn . &\nn . &\nn . &\nn . &\nn . &\nn . \\
  . &\nn . &\nn . &\nn . &\nn . &\nn . &\nn . &\nn . &\nn - &\nn + &\nn . &\nn . &
  . &\nn . &\nn . &\nn . &\nn . &\nn . &\nn . &\nn . &\nn . &\nn . &\nn . &\nn . \\
  . &\nn . &\nn . &\nn . &\nn . &\nn . &\nn . &\nn . &\nn . &\nn . &\nn + &\nn - &
  . &\nn . &\nn . &\nn . &\nn . &\nn . &\nn . &\nn . &\nn . &\nn . &\nn . &\nn . \\
  . &\nn . &\nn . &\nn . &\nn . &\nn . &\nn . &\nn . &\nn . &\nn . &\nn - &\nn + &
  . &\nn . &\nn . &\nn . &\nn . &\nn . &\nn . &\nn . &\nn . &\nn . &\nn . &\nn . \\
  \hline
  + &\nn - &\nn . &\nn . &\nn . &\nn . &\nn + &\nn . &\nn . &\nn . &\nn . &\nn . &
  + &\nn . &\nn . &\nn . &\nn . &\nn . &\nn . &\nn . &\nn . &\nn . &\nn . &\nn . \\
  - &\nn + &\nn . &\nn . &\nn . &\nn . &\nn . &\nn + &\nn . &\nn . &\nn . &\nn . &
  . &\nn + &\nn . &\nn . &\nn . &\nn . &\nn . &\nn . &\nn . &\nn . &\nn . &\nn . \\
  . &\nn . &\nn + &\nn - &\nn . &\nn . &\nn . &\nn . &\nn + &\nn . &\nn . &\nn . &
  . &\nn . &\nn + &\nn . &\nn . &\nn . &\nn . &\nn . &\nn . &\nn . &\nn . &\nn . \\
  . &\nn . &\nn - &\nn + &\nn . &\nn . &\nn . &\nn . &\nn . &\nn + &\nn . &\nn . &
  . &\nn . &\nn . &\nn + &\nn . &\nn . &\nn . &\nn . &\nn . &\nn . &\nn . &\nn . \\
  . &\nn . &\nn . &\nn . &\nn + &\nn - &\nn . &\nn . &\nn . &\nn . &\nn + &\nn . &
  . &\nn . &\nn . &\nn . &\nn + &\nn . &\nn . &\nn . &\nn . &\nn . &\nn . &\nn . \\
  . &\nn . &\nn . &\nn . &\nn - &\nn + &\nn . &\nn . &\nn . &\nn . &\nn . &\nn + &
  . &\nn . &\nn . &\nn . &\nn . &\nn + &\nn . &\nn . &\nn . &\nn . &\nn . &\nn . \\
  . &\nn . &\nn + &\nn + &\nn . &\nn . &\nn . &\nn . &\nn - &\nn . &\nn . &\nn . &
  . &\nn . &\nn . &\nn . &\nn . &\nn . &\nn + &\nn . &\nn . &\nn . &\nn . &\nn . \\
  . &\nn . &\nn . &\nn . &\nn + &\nn + &\nn . &\nn . &\nn . &\nn . &\nn - &\nn . &
  . &\nn . &\nn . &\nn . &\nn . &\nn . &\nn . &\nn + &\nn . &\nn . &\nn . &\nn . \\
  + &\nn + &\nn . &\nn . &\nn . &\nn . &\nn - &\nn . &\nn . &\nn . &\nn . &\nn . &
  . &\nn . &\nn . &\nn . &\nn . &\nn . &\nn . &\nn . &\nn + &\nn . &\nn . &\nn . \\
  . &\nn . &\nn . &\nn . &\nn + &\nn + &\nn . &\nn . &\nn . &\nn . &\nn . &\nn - &
  . &\nn . &\nn . &\nn . &\nn . &\nn . &\nn . &\nn . &\nn . &\nn + &\nn . &\nn . \\
  + &\nn + &\nn . &\nn . &\nn . &\nn . &\nn . &\nn - &\nn . &\nn . &\nn . &\nn . &
  . &\nn . &\nn . &\nn . &\nn . &\nn . &\nn . &\nn . &\nn . &\nn . &\nn + &\nn . \\
  . &\nn . &\nn + &\nn + &\nn . &\nn . &\nn . &\nn . &\nn . &\nn - &\nn . &\nn . &
  . &\nn . &\nn . &\nn . &\nn . &\nn . &\nn . &\nn . &\nn . &\nn . &\nn . &\nn +
  \end{array}
  \right]
  \]
  \caption{The $18 \times 24$ matrix for the proof of Lemma \ref{degree3-Jordan dialgebra}}
  \label{degree3matrix}
  \[
  \left[
  \begin{array}{cccccccccccc|cccccccccccc}
  + &\nn . &\nn . &\nn . &\nn . &\nn . &\nn . &\nn . &\nn . &\nn . &\nn . &\nn . &
  \ast &\nn . &\nn . &\nn . &\nn . &\nn . &\nn . &\nn . &\nn . &\nn . &\nn \ast &\nn . \\
  . &\nn + &\nn . &\nn . &\nn . &\nn . &\nn . &\nn . &\nn . &\nn . &\nn . &\nn . &
  . &\nn \ast &\nn . &\nn . &\nn . &\nn . &\nn . &\nn . &\nn . &\nn . &\nn \ast &\nn . \\
  . &\nn . &\nn + &\nn . &\nn . &\nn . &\nn . &\nn . &\nn . &\nn . &\nn . &\nn . &
  . &\nn . &\nn \ast &\nn . &\nn . &\nn . &\nn . &\nn . &\nn . &\nn . &\nn . &\nn \ast \\
  . &\nn . &\nn . &\nn + &\nn . &\nn . &\nn . &\nn . &\nn . &\nn . &\nn . &\nn . &
  . &\nn . &\nn . &\nn \ast &\nn . &\nn . &\nn . &\nn . &\nn . &\nn . &\nn . &\nn \ast \\
  . &\nn . &\nn . &\nn . &\nn + &\nn . &\nn . &\nn . &\nn . &\nn . &\nn . &\nn . &
  . &\nn . &\nn . &\nn . &\nn \ast &\nn . &\nn . &\nn . &\nn . &\nn \ast &\nn . &\nn . \\
  . &\nn . &\nn . &\nn . &\nn . &\nn + &\nn . &\nn . &\nn . &\nn . &\nn . &\nn . &
  . &\nn . &\nn . &\nn . &\nn . &\nn \ast &\nn . &\nn . &\nn . &\nn \ast &\nn . &\nn . \\
  . &\nn . &\nn . &\nn . &\nn . &\nn . &\nn + &\nn . &\nn . &\nn . &\nn . &\nn . &
  \ast &\nn \ast &\nn . &\nn . &\nn . &\nn . &\nn . &\nn . &\nn . &\nn . &\nn . &\nn . \\
  . &\nn . &\nn . &\nn . &\nn . &\nn . &\nn . &\nn + &\nn . &\nn . &\nn . &\nn . &
  \ast &\nn \ast &\nn . &\nn . &\nn . &\nn . &\nn . &\nn . &\nn . &\nn . &\nn . &\nn . \\
  . &\nn . &\nn . &\nn . &\nn . &\nn . &\nn . &\nn . &\nn + &\nn . &\nn . &\nn . &
  . &\nn . &\nn \ast &\nn \ast &\nn . &\nn . &\nn . &\nn . &\nn . &\nn . &\nn . &\nn . \\
  . &\nn . &\nn . &\nn . &\nn . &\nn . &\nn . &\nn . &\nn . &\nn + &\nn . &\nn . &
  . &\nn . &\nn \ast &\nn \ast &\nn . &\nn . &\nn . &\nn . &\nn . &\nn . &\nn . &\nn . \\
  . &\nn . &\nn . &\nn . &\nn . &\nn . &\nn . &\nn . &\nn . &\nn . &\nn + &\nn . &
  . &\nn . &\nn . &\nn . &\nn \ast &\nn \ast &\nn . &\nn . &\nn . &\nn . &\nn . &\nn . \\
  . &\nn . &\nn . &\nn . &\nn . &\nn . &\nn . &\nn . &\nn . &\nn . &\nn . &\nn + &
  . &\nn . &\nn . &\nn . &\nn \ast &\nn \ast &\nn . &\nn . &\nn . &\nn . &\nn . &\nn . \\
  \hline
  . &\nn . &\nn . &\nn . &\nn . &\nn . &\nn . &\nn . &\nn . &\nn . &\nn . &\nn . &
  . &\nn . &\nn . &\nn . &\nn . &\nn . &\nn + &\nn . &\nn . &\nn . &\nn . &\nn - \\
  . &\nn . &\nn . &\nn . &\nn . &\nn . &\nn . &\nn . &\nn . &\nn . &\nn . &\nn . &
  . &\nn . &\nn . &\nn . &\nn . &\nn . &\nn . &\nn + &\nn . &\nn - &\nn . &\nn . \\
  . &\nn . &\nn . &\nn . &\nn . &\nn . &\nn . &\nn . &\nn . &\nn . &\nn . &\nn . &
  . &\nn . &\nn . &\nn . &\nn . &\nn . &\nn . &\nn . &\nn + &\nn . &\nn - &\nn .
  \end{array}
  \right]
  \]
  \caption{The row canonical form of the matrix of Figure \ref{degree3matrix}}
  \label{degree3matrixrcf}
  \end{figure}
  
\begin{theorem} \label{degree5-Jordan dialgebra}
In a Jordan dialgebra $D$ with operation $ab$, 
every polynomial identity of degree 5 satisfied by 
$\langle a,b,c \rangle_1$ and $\langle a,b,c \rangle_2$
is a consequence of the identity of Lemma \ref{degree3-Jordan dialgebra}
together with these seven independent identities:
\begin{align*}
&
\langle a, \langle b,c,d \rangle_1, e \rangle_1 - \langle a, \langle d,c,b \rangle_2, e \rangle_1 
\equiv 0,
\\
&
\langle a, b, \langle c,d,e \rangle_1 \rangle_1 - \langle a, b, \langle e,d,c\rangle_2 \rangle_1 
\equiv 0, 
\\
&
\langle \langle a,b,c \rangle_1, d, e \rangle_2 - \langle \langle c,b,a \rangle_2, d, e \rangle_2 
\equiv 0, 
\\
&
\langle \langle a,b,c \rangle_2, d, e \rangle_1 - \langle \langle e,d,a \rangle_2, b, c \rangle_2 
+ 
\langle a, \langle b,e,d \rangle_1, c \rangle_2 - \langle \langle e,d,c \rangle_2, b, a \rangle_2 
\equiv 0,
\\
&
\langle \langle a,b,c \rangle_2, d, e \rangle_1 - \langle a, \langle b,c,d \rangle_1, e \rangle_2 
- 
\langle e, \langle b,a,d \rangle_1, c \rangle_2 + \langle \langle a,d,c \rangle_2, b, e \rangle_2 
\equiv 0,
\\
&
\langle \langle a,b,c \rangle_1, d, e \rangle_1 + \langle \langle a,b,e \rangle_1, d, c \rangle_1 
- 
\langle a, b, \langle c,d,e \rangle_2 \rangle_1 - \langle c, \langle d,a,b \rangle_2, e \rangle_2 
\equiv 0, 
\\
&
\langle \langle a,b,c \rangle_1, d, e \rangle_1 + \langle \langle a,d,c \rangle_1, b, e \rangle_1 
- 
\langle a, \langle b,c,d \rangle_2, e \rangle_1 - \langle e, \langle b,a,d \rangle_2, c \rangle_2 
\equiv 0.
\end{align*}
\end{theorem}

\begin{proof}
The basic strategy is the same as in the proof of Lemma \ref{degree3-Jordan dialgebra}, 
but the matrix is much larger and some further computations are required.

There are 14 association types for a nonassociative binary operation of degree 5,
  \begin{align*}
  &
  (((ab)c)d)e, ((a(bc))d)e, ((ab)(cd))e, (a((bc)d))e, (a(b(cd)))e, ((ab)c)(de), (a(bc))(de),
  \\
  &
  (ab)((cd)e), (ab)(c(de)), a(((bc)d)e), a((b(cd))e), a((bc)(de)), a(b((cd)e)), a(b(c(de))).
  \end{align*}
and $5!$ permutations of the variables $a,b,c,d,e$, giving 1680 multilinear monomials,
which correspond to the columns in the left part of the matrix;
we order them first by association type and then by lexicographical order of the permutation.

We need to generate all the consequences of degree 5 of the defining identities for Jordan dialgebras.
A multilinear identity $I(a_1,\dots,a_n)$ of degree $n$ produces $n{+}2$ identities of degree $n{+}1$
(we have $n$ substitutions and two multiplications):
  \[
  I( a_1 a_{n+1}, \dots, a_n ),
  \;
  \dots,
  \;
  I( a_1, \dots, a_n a_{n+1} ),
  \;
  I( a_1, \dots, a_n ) a_{n+1},
  \;
  a_{n+1} I( a_1, \dots, a_n ).
  \]
The right-commutative identity of degree 3 produces 5 identities of degree 4, and each of these 
produces 6 identities of degree 5, for a total of 30.
The linearized versions of the two identities of degree 4 from Definition \ref{JDdefinition} each
produce 6 identities of degree 5, for a total of 12.
Altogether we have 42 identities of degree 5, and each allows $5!$ permutations of the variables, 
for a total of 5040 identities.

The last two paragraphs show that the upper left block has size $5040 \times 1680$;
its $(i,j)$ entry is the coefficient of the $j$-th nonassociative monomial in the $i$-th identity.

There are 10 association types for two trilinear operations of degree 5,
assuming that the second operation is symmetric in its first and third arguments:
\begin{center}
\begin{tabular}{rlrl}
 1: &\quad $\langle \langle a,b,c\rangle_1,d,e\rangle_1$ &\qquad
 6: &\quad $\langle \langle a,b,c\rangle_2,d,e\rangle_1$ \\
 2: &\quad $\langle a,\langle b,c,d\rangle_1,e\rangle_1$ &\qquad
 7: &\quad $\langle a,\langle b,c,d\rangle_2,e\rangle_1$ \\
 3: &\quad $\langle a,b,\langle c,d,e\rangle_1\rangle_1$ &\qquad
 8: &\quad $\langle a,b,\langle c,d,e\rangle_2\rangle_1$ \\
 4: &\quad $\langle \langle a,b,c\rangle_2,d,e\rangle_2$ &\qquad
 9: &\quad $\langle \langle a,b,c\rangle_1,d,e\rangle_2$ \\
 5: &\quad $\langle a,\langle b,c,d\rangle_2,e\rangle_2$ &\qquad
10: &\quad $\langle a,\langle b,c,d\rangle_1,e\rangle_2$
\end{tabular}
\end{center}
The symmetry of $\langle \cdots \rangle_2$ gives the number of multilinear monomials in each type:
  \[
  120 + 120 + 120 + 60 + 60 + 60 + 120 + 60 + 60 + 30 = 810;
  \]
these monomials correspond to the columns in the right part of the matrix.
The lower left block has size $810 \times 1680$; 
its $(i,j)$ entry is the coefficient of the $j$-th nonassociative monomial in the expansion, 
using equations \eqref{diop1} and \eqref{diop2},
of the $i$-th monomial in the operations 
$\langle \cdots \rangle_1$ and $\langle \cdots \rangle_2$.
The lower right block is the $810 \times 810$ identity matrix, and
the upper right block is the $5040 \times 810$ zero matrix.
(See Figure \ref{figureJDmatrix}.)
 
  \begin{figure}
  \[
  \left[
  \begin{array}{c|c}
  \begin{tabular}{l}
  consequences in \\
  degree 5 of the \\
  Jordan dialgebra \\
  identities from \\
  Definition \ref{JDdefinition}
  \end{tabular}
  &
  \begin{tabular}{l}
  zero matrix
  \end{tabular}
  \\
  \hline
  \begin{tabular}{l}
  expansions using \\
  \eqref{diop1} and \eqref{diop2} of \\  
  the monomials \\
  of degree 5 for \\
  $\langle \cdots \rangle_1$ and $\langle \cdots \rangle_2$ \\
  \end{tabular}
  &
  \begin{tabular}{l}
  identity matrix
  \end{tabular}
  \end{array}
  \right]
  \]
  \caption{Matrix for the proof of Theorem \ref{degree5-Jordan dialgebra}}
  \label{figureJDmatrix}
  \end{figure}
  
We compute the row canonical form of this matrix and find that its rank is 2215.
We ignore the first 1655 rows since their leading ones are in the left part, and 
retain only the next 560 rows which have have their leading ones in the right part.
We sort these rows by increasing number of nonzero components. 

We construct another matrix with an upper block of size $810 \times 810$ and 
a lower block of size $120 \times 810$.
For each of the 560 identities satisfied by the operations 
$\langle \cdots \rangle_1$ and $\langle \cdots \rangle_2$,
we apply all $5!$ permutations of the variables, store the permuted identities in the lower block,
and compute the row canonical form of the matrix; 
after each iteration, the lower block is the zero matrix.
We record the index numbers of the identities which increase the rank, and obtain these results:
  \begin{center}
  \begin{tabular}{lrrrrrrrr}
  identity & 1 & 121 & 241 & 301 & 331 & 342 & 451 & 454 \\
  rank & 120 & 240 & 360 & 390 & 450 & 470 & 530 & 560 
  \end{tabular}
  \end{center}
Further computations show that identity 301 is superfluous: 
the other seven identities generate the entire 560-dimensional space.
These seven identities are displayed in standard notation in the statement of this Theorem.
\end{proof}

\begin{theorem} \label{Jordandialgebratheorem} 
If $D$ is a subspace of a Jordan dialgebra over a field of characteristic not $2, 3, 5$
which is closed under the trilinear operations $\langle\cdots\rangle_1$ and $\langle\cdots\rangle_2$
of equations \eqref{diop1} and \eqref{diop2},
then $D$ is a Jordan triple disystem with respect to these operations.
\end{theorem}

\begin{proof}
Theorem \ref{degree3-Jordan dialgebra} tells us that identity \eqref{JTD1} is satisfied:
  \[
  \langle a, b, c\rangle_2 \equiv  \langle c, b, a\rangle_2.
  \] 
Identities \eqref{JTD2}, \eqref{JTD3} and \eqref{JTD4} can be easily obtained from \eqref{JTD1}
and the first three identities of Theorem \ref{degree5-Jordan dialgebra}:
  \begin{align*}
  \langle a, \langle b, c, d\rangle_1, e\rangle_1 
  &\equiv 
  \langle a, \langle d, c, b\rangle_2, e\rangle_1 
  \equiv 
  \langle a, \langle b, c, d\rangle_2, e\rangle_1, 
  \\
  \langle a, b, \langle c, d, e\rangle_1\rangle_1 
  &\equiv 
  \langle a, b, \langle e, d, c\rangle_2\rangle_1 
  \equiv 
  \langle a, b, \langle c, d, e\rangle_2\rangle_1,
  \\
  \langle \langle a, b, c\rangle_1, d, e\rangle_2 
  &\equiv 
  \langle \langle c, b, a\rangle_2, d, e\rangle_2 
  \equiv 
  \langle \langle c, b, a\rangle_2, d, e\rangle_2.
  \end{align*}
Identity \eqref{JTD7} is a consequence of the sixth identity of Theorem \ref{degree5-Jordan dialgebra} 
by applying \eqref{JTD1} and \eqref{JTD3}. 
To obtain \eqref{JTD6}, we apply the transpositions $ae$ and $bd$ to the fourth identity of 
Theorem \ref{degree5-Jordan dialgebra}: 
  \[
  \langle \langle e,d,c\rangle_2, b, a \rangle_1 
  - 
  \langle \langle a,b,e\rangle_2, d, c \rangle_2 
  + 
  \langle e, \langle d,b,a\rangle_1, c \rangle_2 
  - 
  \langle \langle a,b,c\rangle_2, d, e \rangle_2 
  \equiv 
  0.
  \] 
Using identities \eqref{JTD1} and \eqref{JTD4} we obtain these two identities, 
  \begin{align*}
  \langle \langle a, b, e\rangle_2, d, c\rangle_2 
  &\equiv 
  \langle \langle e, b, a\rangle_2, d, c\rangle_2 
  \equiv 
  \langle \langle e, b, a\rangle_1, d, c\rangle_2, 
  \\
  \langle \langle a, b, c\rangle_2, d, e\rangle_2 
  &\equiv 
  \langle \langle c, b, a\rangle_2, d, e\rangle_2 
  \equiv 
  \langle \langle c, b, a\rangle_1, d, e\rangle_2,
  \end{align*} 
and substituting these in the previous identity gives \eqref{JTD6}. 
To obtain identity \eqref{JTD8}, we apply \eqref{JTD1} and \eqref{JTD4} to the fifth identity of 
Theorem \ref{degree5-Jordan dialgebra}: 
  \begin{equation} \label{before JTD6}
  \langle \langle a, b, c \rangle_2, d, e \rangle_1 
  - 
  \langle a, \langle b, c, d \rangle_1, e \rangle_2 
  - 
  \langle c, \langle b, a, d \rangle_1, e \rangle_2 
  + 
  \langle \langle a, d, c \rangle_1, b, e \rangle_2 
  \equiv 
  0.
  \end{equation} 
Identity \eqref{JTD6} gives  
  \[
  - 
  \langle c, \langle b, a, d \rangle_1, e \rangle_2 
  + 
  \langle \langle a, d, c \rangle_1, b, e \rangle_2 
  \equiv 
  - 
  \langle a, b, \langle e, d, c \rangle_1 \rangle_2 
  + 
  \langle \langle a, b, e \rangle_2, d, c \rangle_1.
  \] 
On the other hand, using identities \eqref{JTD1} and \eqref{JTD4} we obtain
  \begin{align*}
  &
  \langle a, b, \langle e, d, c\rangle_1\rangle_2 
  \equiv 
  \langle \langle e, d, c\rangle_1, b, a\rangle_2 
  \equiv 
  \langle \langle e, d, c\rangle_2, b, a\rangle_2 
  \equiv {}
  \\
  &
  \langle \langle c, d, e\rangle_2, b, a\rangle_2 
  \equiv 
  \langle \langle c, d, e\rangle_1, b, a\rangle_2 
  \equiv 
  \langle a, b, \langle c, d, e\rangle_1\rangle_2.
  \end{align*} 
Applying these considerations to identity \eqref{before JTD6} we obtain
  \[
  \langle \langle a, b, c \rangle_2, d, e \rangle_1 
  - 
  \langle c, \langle b, a, d \rangle_1, e \rangle_2 
  + 
  \langle \langle a, b, e \rangle_2, d, c \rangle_1 
  - 
  \langle a, b, \langle c, d, e \rangle_1 \rangle_2 
  \equiv 
  0,
  \]
which is equivalent to \eqref{JTD8}. 
It remains to show \eqref{JTD5}.
We apply \eqref{JTD1} and \eqref{JTD2} to the last identity of Theorem \ref{degree5-Jordan dialgebra} and obtain
  \begin{equation} \label{before JTD5}
  \langle \langle a, b, c \rangle_1, d, e \rangle_1 
  + 
  \langle \langle a, d, c \rangle_1, b, e \rangle_1 
  - 
  \langle a, \langle b, c, d \rangle_1, e \rangle_1 
  - 
  \langle c, \langle b, a, d \rangle_2, e \rangle_2 
  \equiv 
  0.
  \end{equation} 
From \eqref{JTD7} we have: 
  \[
  \langle c, \langle b, a, d \rangle_2, e \rangle_2 
  \equiv 
  - 
  \langle a, b, \langle c, d, e \rangle_1 \rangle_1 
  + 
  \langle \langle a, b, c \rangle_1, d, e \rangle_1 
  + 
  \langle \langle a, b, e \rangle_1, d, c \rangle_1, 
  \]
and substituting this in \eqref{before JTD5} gives
  \begin{equation} \label{almost JTD5}
  \langle \langle a, d, c \rangle_1, b, e \rangle_1 
  - 
  \langle a, \langle b, c, d \rangle_1, e \rangle_1 
  + 
  \langle a, b, \langle c, d, e \rangle_1 \rangle_1 
  - 
  \langle \langle a, b, e \rangle_1, d, c \rangle_1 
  \equiv 
  0.
  \end{equation} 
To finish, we observe that 
  \[ 
  \langle a, b, \langle c, d, e \rangle_1 \rangle_1  
  \equiv 
  \langle a, b, \langle c, d, e \rangle_2 \rangle_1  
  \equiv 
  \langle a, b, \langle e, d, c \rangle_2 \rangle_1 
  \equiv 
  \langle a, b, \langle e, d, c \rangle_1 \rangle_1,
  \]
and applying these to \eqref{almost JTD5} yields
  \[
  \langle \langle a, d, c \rangle_1, b, e \rangle_1 
  - 
  \langle a, \langle b, c, d \rangle_1, e \rangle_1 
  + 
  \langle a, b, \langle e, d, c \rangle_1 \rangle_1 
  - 
  \langle \langle a, b, e \rangle_1, d, c \rangle_1 
  \equiv 
  0,
  \]
which is equivalent to \eqref{JTD5}. 
\end{proof}

\subsection*{Second proof of Theorem \ref{Jordandialgebratheorem}}

We conclude with an alternative proof, without using computer algebra,  
of the defining identities \eqref{JTD1}--\eqref{JTD8} for Jordan triple disystems 
with respect to the operations $\langle \cdots \rangle_1$ and $\langle \cdots \rangle_2$
in a Jordan dialgebra:
  \[
  \langle a,b,c \rangle_1 = (ab)c - (ac)b + a(bc),
  \qquad
  \langle a,b,c \rangle_2 = (ba)c + (bc)a - b(ac).
  \]

\begin{definition}
Let $D$ be a Jordan dialgebra with operation $ab$.
The \textbf{annihilator} $D^\mathrm{ann}$ is the ideal spanned by 
$\{ \, [a,b] = ab-ba \mid a, b \in D \, \}$.
The \textbf{right center} $Z^{r}(D)$ is the ideal consisting of the elements $b \in D$ 
such that $ab = 0$ for all $a \in D$. 
A \textbf{derivation} is a linear map $\delta\colon D \rightarrow D$ satisfying 
$\delta(ab) = \delta(a)b + a\delta(b)$ for all $a, b \in D$.
A \textbf{left derivation} is a linear map $\mu\colon D \rightarrow D$ satisfying 
$\mu(ab) = \mu(a)b + \mu(b)a$ for all $a, b \in D$. 
For $a \in D$ the \textbf{right} and \textbf{left multiplication operators} 
$R_a\colon D \rightarrow D$, $L_a\colon D \rightarrow D$ are defined by 
$R_a b = ba$, $L_a b = ab$ for all $b \in D$. 
\end{definition}

\begin{lemma}
We have $D^{\rm ann} \subseteq Z^{r}(D)$ for any Jordan dialgebra $D$.
\end{lemma}

\begin{proof}
Right commutativity gives $c[a,b] = c(ab) - c(ba) = 0$.
\end{proof}

\begin{lemma} \label{difference}
If $D$ is a Jordan dialgebra then for all $a, b, c \in D$ we have
  \[
  \langle a,b,c \rangle_1 - \langle c,b,a \rangle_1 \in D^{\rm ann},
  \qquad
  \langle a,b,c \rangle_1 - \langle a,b,c \rangle_2 \in D^{\rm ann}.
  \]
\end{lemma}

\begin{proof} 
Right commutativity gives
  \begin{align*}
  &
  \langle a,b,c \rangle_1 - \langle c,b,a \rangle_1 
  = 
  (ab)c - (ac)b + a(bc) - (cb)a + (ca)b - c(ba) 
  \\ 
  &= 
  (ab)c - c(ab) + (ca-ac)b + a(cb) - (cb)a
  =
  [ab,c] + [c,a]b + [a,cb],
  \\
  &
  \langle a,b,c \rangle_1 - \langle a,b,c \rangle_2 
  = 
  (ab)c-(ac)b+ a(bc)-(ba)c -(bc)a+b(ac) 
  \\
  &= 
  (ab-ba)c + a(bc) - (bc)a + b(ac) - (ac)b
  =
  [a,b]c + [a,bc] + [b,ac].
  \end{align*} 
Both expressions belong to $D^\mathrm{ann}$ since the annihilator is an ideal.
\end{proof}

\begin{proposition} \label{Identities from 1 to 4}
If $D$ is a Jordan dialgebra then identities \eqref{JTD1}--\eqref{JTD4} are satisfied 
by the operations $\langle \cdots \rangle_1$ and $\langle \cdots \rangle_2$:
  \begin{align*}
  &
  \langle a,b,c \rangle_2 \equiv \langle c,b,a \rangle_2,
  \\ 
  &
  \langle a, \langle b, c, d \rangle_1, e \rangle_1 \equiv
  \langle a, \langle b, c, d \rangle_2, e \rangle_1 \equiv
  \langle a, \langle d, c, b \rangle_1, e \rangle_1,
  \\ 
  &
  \langle a, b, \langle c, d, e \rangle_1 \rangle_1 \equiv
  \langle a, b, \langle c, d, e \rangle_2 \rangle_1 \equiv
  \langle a, b, \langle e, d, c \rangle_1 \rangle_1,
  \\ 
  &
  \langle \langle a, b, c \rangle_1, d, e \rangle_2 \equiv
  \langle \langle a, b, c \rangle_2, d, e \rangle_2 \equiv
  \langle \langle c, b, a \rangle_1, d, e \rangle_2.
  \end{align*}
\end{proposition}

\begin{proof} 
The first line follows directly from right commutativity. 
For the second line, using Lemma \ref{difference} it suffices to show
$\langle a, x, e \rangle_1 \equiv 0$ for $x \in D^\mathrm{ann}$,
and again this follows from right commutativity.
The third and fourth lines are similar.
\end{proof}

Right commutativity is equivalent to both $R_{ab} = R_{ba}$ and $L_a L_b = L_a R_b$. 
The operations $\langle \cdots \rangle_1$ and $\langle \cdots \rangle_2$ 
can be expressed using multiplication operators: 
  \begin{align*}
  \langle a,b,c \rangle_1 
  &=
  R_c R_b \, a - R_b R_c \, a + R_{cb} \, a  
  = 
  \big( \, R_{cb} + [ R_c, R_b ] \, \big) a
  \\
  &=
  L_{ab} \, c - R_b L_a \, c + L_a L_b \, c 
  = 
  \big( \, L_{ab} + [ L_a, R_b ] \, \big) c,
  \\
  \langle a,b,c \rangle_2 
  &=
  L_{ba} \, c + R_a L_b \, c - L_b L_a \, c 
  = 
  \big( L_{ba} - [ L_b, R_a ] \, \big) c.
  \end{align*}
It has been shown by Felipe and Vel\'asquez \cite{Felipe, VelasquezFelipe3} that 
$[R_a, R_b]$ is a derivation, that $[L_a, R_b]$ is a left derivation, and 
that for all $a, b, c \in D$ we have
  \begin{align}
  R_{[R_a,R_b]c} &= [[R_a,R_b],R_c], 
  \label{R and bracket} 
  \\
  L_{[L_a,R_b]c} &= [[L_a,R_b],R_c].
  \label{L and bracket}  
  \end{align}

\begin{lemma} \label{R_t, L_t and products 1, 2} 
If $D$ is a Jordan dialgebra then for $i = 1, 2$ and $a, b, c, d \in D$ we have
  \begin{align*}
  (1) \quad 
  R_d \, \langle a, b, c \rangle_i  
  &= 
  \langle \, R_d \, a, b, c \, \rangle_i 
  - 
  \langle \, a, R_d \, b, c \, \rangle_i 
  + 
  \langle \, a, b, R_d \, c \, \rangle_i,
  \\
  (2) \quad
  L_d \, \langle a, b, c \rangle_i
  &= 
  \langle \, L_d \, a, b, c \, \rangle_1 
  - 
  \langle \, a, R_d \, b, c \, \rangle_2 
  + 
  \langle \, L_d \, c, b, a \, \rangle_1.
  \end{align*}
\end{lemma}

\begin{proof}
(1) For $i = 1$ we use the linearization of the right Jordan identity,
  \begin{equation} \label{double linearization}
  [R_{a},R_{bc}] + [R_{b},R_{ca}] + [R_{c},R_{ab}] = 0.
  \end{equation} 
By definition of $\langle \cdots \rangle_1$ we have
  \begin{align*} 
  R_d \langle a, b, c \rangle_1 &= ((ab)c)d - ((ac)b)d + (a(bc))d, 
  \\
  \langle R_d a, b, c \rangle_1 &= ((ad)b)c - ((ad)c)b + (ad)(bc),
  \\
  \langle a, R_d b, c \rangle_1 &= (a(bd))c - (ac)(bd) + a((bd)c),
  \\
  \langle a, b, R_d c \rangle_1 &= (ab)(cd) - (a(cd))b + a(b(cd)).
  \end{align*}
Using right commutativity and equations \eqref{R and bracket} and \eqref{double linearization}
we obtain
  \begin{align*} 
  &
  R_d \langle a,b,c \rangle_1 
  - 
  \langle R_d a,b,c \rangle_1 
  + 
  \langle a,R_d b,c \rangle_1 
  - 
  \langle a,b,R_d c \rangle_1 
  \\
  &= 
  (a(bc))d - (ad)(bc) + (a(cd))b - (ab)(cd) + (a(bd))c - (ac)(bd)
  \\
  &\quad
  +
  ((ab)c)d - ((ac)b)d - ((ad)b)c + ((ad)c)b + a ((bd)c) - a(b(cd))
  \\
  &=
  \big(
  [R_d, R_{cb}]+ [R_b, R_{cd}] + [R_c, R_{bd}] + [R_d,[R_c, R_b]] + R_{[R_c, R_b]d}
  \big) 
  a 
  = 
  0.
  \end{align*} 
For $i = 2$ the proof is similar, using \eqref{L and bracket} and this equation proved in \cite{VelasquezFelipe3}:
  \begin{equation} \label{L, R, Lie bracket}
  [L_{a},R_{bc}] + [R_{b},L_{ac}] + [R_{c},L_{ab}] = 0, 
  \end{equation} 
 
(2) For $i = 1$, right commutativity and the definitions of 
$\langle \cdots \rangle_1$ and $\langle \cdots \rangle_2$
give
  \begin{align*} 
  &
  L_d \langle a,b,c \rangle_1 
  - 
  \langle L_d a,b,c \rangle_1 
  + 
  \langle a,L_d b,c \rangle_2 
  - 
  \langle L_d c,b,a \rangle_1 
  \\
  &=
  \big(
  L_d L_{ab} - L_d R_b L_a + L_d L_a L_b 
  - L_{R_b R_a d} + R_b L_{da} - L_{da} R_b 
  \\
  &\quad
  + L_{R_a R_b d} + R_a L_{db} - L_{db} L_a 
  - R_a R_b L_d + R_b R_a L_d - R_{ba} L_d
  \big) 
  c
  \\
  & 
  = 
  \big(
  L_d R_{ab} - L_d R_b R_a + L_d R_a R_b 
  + L_{[R_a, R_b] d} + R_b L_{da} - L_{da}R_b 
  \\
  &\quad
  + R_a L_{db} - L_{db} R_a - [R_a, R_b]ÊL_d - R_{ab} L_d 
  \big) 
  c 
  \\
  &=
  \big(
  [L_d,R_{ab}] + [R_b,L_{da}] + [R_a,L_{db}] +  L_{[R_a, R_b] d}\big) 
  c = 0.
  \end{align*}
We have used \eqref{L, R, Lie bracket} and \eqref{L and bracket} and the fact that
  \[
  L_{[R_a, R_b]d} = [[R_a, R_b], L_d],
  \]
since $[R_a, R_b]$ is a derivation of $D$.
For $i = 2$, it suffices to observe that right commutativity implies 
$L_d \langle a, b, c \rangle_1 = L_d \langle a, b, c \rangle_2$.
\end{proof}

\begin{lemma} \label{derivations}
If $D$ is a Jordan dialgebra with derivation $\delta$ and left derivation $\mu$
then for $i = 1, 2$ and $a, b, c, d \in D$ we have
  \begin{align*}
  \delta \langle a, b, c \rangle_i
  &= 
  \langle \delta a, b, c \rangle_i + \langle a, \delta b, c \rangle_i + \langle a, b, \delta c \rangle_i,
  \\
  \mu \langle a, b, c \rangle_i
  &= 
  \langle \mu a, b, c \rangle_1 + \langle a, \mu b, c \rangle_2 + \langle \mu c, b, a \rangle_1.
  \end{align*}
\end{lemma}

\begin{proof}
Straightforward.
\end{proof}

\begin{proposition}
If $D$ is a Jordan dialgebra then identities \eqref{JTD5}--\eqref{JTD8} are satisfied by 
the operations $\langle \cdots \rangle_1$ and $\langle \cdots \rangle_2$.
\end{proposition}

\begin{proof}
For $a, b, c, d, e \in D$ we apply Lemmas \ref{R_t, L_t and products 1, 2} (1) and \ref{derivations} to get
  \begin{align*}
  &
  \langle \langle e, d, c \rangle_1, b, a \rangle_1 
  =  
  \big( R_{ab} + [R_a, R_b] \big) \langle e, d, c \rangle_1 
  \\
  &= 
  \langle ( R_{ab} {+} [R_a, R_b] ) e, d, c \rangle_1 
  - 
  \langle e, ( R_{ba} {+} [R_b, R_a] ) d, c \rangle_1 
  + 
  \langle e, d, ( R_{ab} {+} [R_a, R_b] ) c \rangle_1 
  \\
  &= 
  \langle \langle e, b, a \rangle_1, d, c \rangle_1 
  - 
  \langle e, \langle d, a, b \rangle_1, c \rangle_1 
  + 
  \langle e, d, \langle c, b, a \rangle_1 \rangle_1,
  \end{align*}
which is identity \eqref{JTD5}. 
The proof of identity \eqref{JTD6} is similar, 
replacing the outermost operation $\langle \cdots \rangle_1$ by $\langle \cdots \rangle_2$. 
Finally, Lemmas \ref{R_t, L_t and products 1, 2} (2) and \ref{derivations} give
  \begin{align*}
  &
  \langle a, b, \langle c, d, e \rangle_1 \rangle_1 
  =  \big( L_{ab} + [L_a, R_b] \big) \langle c, d, e \rangle_1 
  \\
  &= 
  \langle ( L_{ab} {+} [L_a, R_b] ) c, d, e \rangle_1 
  - 
  \langle c, ( L_{ab} {-} [L_a, R_b] ) d, e \rangle_2 
  + 
  \langle ( L_{ab} {+} [L_a, R_b] ) e, d, c \rangle_1 
  \\
  &= 
  \langle \langle a, b, c \rangle_1, d, e \rangle_1 
  - 
  \langle c, \langle b, a, d \rangle_2, e \rangle_2 
  + 
  \langle \langle a, b, e \rangle_1, d, c \rangle_1,
  \\
  &
  \langle a,b,\langle c, d, e \rangle_1 \rangle_2 
  =  
  \big( L_{ba}- [L_b, R_a] \big) \langle c, d, e \rangle_1 
  \\
  &= 
  \langle ( L_{ba} {-} [L_b, R_a] )c, d, e \rangle_1 
  - 
  \langle c, ( L_{ba} {+} [L_b, R_a] )d, e \rangle_2 
  + 
  \langle ( L_{ba} {-} [L_b, R_a] )e, d, c \rangle_1 
  \\
  &= 
  \langle \langle a, b, c \rangle_2, d, e \rangle_1 
  - 
  \langle c, \langle b, a, d \rangle_1, e \rangle_2 
  + 
  \langle \langle a, b, e \rangle_2, d, c \rangle_1,
  \end{align*}
which are identities \eqref{JTD7} and \eqref{JTD8}.
The proof is complete.
\end{proof}


\section*{Acknowledgements}

Murray Bremner was supported by a Discovery Grant from NSERC;
he thanks the Centro de Investigaci\'on en Matem\'aticas in Guanajuato (Mexico) 
for its hospitality during his visit in February 2011.
Ra\'ul Felipe was supported by CONACyT grant 106923.
Juana S\'anchez-Ortega was supported by the Spanish MEC and Fondos FEDER jointly through project
MTM2010-15223, and by the Junta de Andaluc\'ia (projects FQM-336 and FQM2467).
She thanks the Department of Mathematics and Statistics at the University of Saskatchewan
(Canada) for its hospitality during her visit from March to June 2011.


\end{document}